\documentclass[a4paper, 12pt,reqno]{amsart}
\usepackage[12pt]{extsizes}
\usepackage{amsmath,amssymb}
\usepackage{amscd}
\usepackage{geometry}
\usepackage{amsthm}
\usepackage{euscript}
\usepackage{latexsym}
\usepackage[matrix,arrow,curve,cmtip]{xy}
\usepackage[cp1251]{inputenc}
\usepackage[russian,english]{babel}
\usepackage{wrapfig}
\usepackage{mathrsfs}
\usepackage{graphicx}
\usepackage{keyval}
\usepackage{mathtools}
\usepackage[x11names]{xcolor}
\usepackage{tikz}
\usetikzlibrary{arrows,shapes,snakes,automata,backgrounds,petri,through,positioning}
\usetikzlibrary{intersections}
\usepackage[symbol*]{footmisc}
\usepackage{verbatim}
\usepackage{fancyhdr}

\usepackage{mathptmx} 
\DeclareMathAlphabet{\mathcal}{OMS}{cmsy}{m}{n} 

\usepackage{amssymb,amscd,enumerate,mathrsfs,silence}

\newtheorem{theorem}{Theorem}[section]
\newtheorem{lemma}{Lemma}[section]
\newtheorem{proposition}{Proposition}[section]
\newtheorem{corollary}{Corollary}[section]

\theoremstyle{definition}
\newtheorem{definition}[theorem]{Definition}
\newtheorem{example}[theorem]{Example}
\newtheorem{remark}[theorem]{Remark}

\newcommand{\m}{\mathfrak}

\begin{document}

\pagestyle{plain}
\title{\bf ON REALIZABILITY OF GAUSS DIAGRAMS and CONSTRUCTIONS OF MEANDERS}

\maketitle

\begin{center}
Andrey Grinblat\footnote{\texttt{expandrey@mail.ru}} and Viktor Lopatkin\footnote{\texttt{wickktor@gmail.com} (please use this e-mail for contacting.)}
\end{center}

\begin{abstract}
  The problem of which Gauss diagram can be realized by knots is an old one and has been solved in several ways. In this paper, we present a direct approach to this problem. We show that the needed conditions for realizability of a Gauss diagram can be interpreted as follows ``the number of exits = the number of entrances'' and the sufficient condition is based on Jordan curve Theorem. Further, using matrixes we redefine conditions for realizability of Gauss diagrams and then we give an algorithm to construct meanders.

  \medskip


  \end{abstract}

\section*{Introduction}
In the earliest time of the Knot Theory C.F. Gauss defined the chord diagram (= Gauss diagram). C.F. Gauss \cite{Gauss} observed that if a chord diagram can be realized by a plane curve, then every chord is crossed only by an even number of chords, but that this condition is not sufficient.

The aim of this paper is to present a direct approach to the problem of which Gauss diagram can be realized by knots. This problem is an old one, and has been solved in several ways.

In 1936, M. Dehn \cite{D} found a sufficient algorithmic solution based on the existence of a touch Jordan curve which is the image of a transformation of the knot diagram by successive splits replacing all the crossings. A long time after in 1976, L. Lovasz and M.L. Marx \cite{LM} found a second necessary condition and finally during the same year, R.C. Read and P. Rosenstiehl \cite{RR} found the third condition which allowed the set of these three conditions to be sufficient. The last characterization is based on the tripartition of graphs into cycles, cocycles and bicycles.

In \cite{S1} the notation of oriented chord diagram was introduced and it was showed that these diagrams classify cellular generic curves on oriented surfaces. As a corollary a simple combinatorial classification of plane generic curves was derived, and the problem of realizability of these diagrams was also solved.

However all these ways are indirect; they rest upon deep and nontrivial auxiliary construction. There is a natural question: whether one can arrive at these conditions in a more direct and natural fashion?

We believe that the conditions for realizability of a Gauss diagram (by some plane curve) should be obtained in a natural manner; they should be deduced from an intrinsic structure of the curve.

In this paper, we suggest an approach, which satisfies the above principle. We use the fact that every Gauss diagram $\m{G}$ defines a (virtual) plane curve $\mathscr{C}(\m{G})$ (see \cite[Theorem 1.A]{GPV}), and the following simple ideas:
\begin{itemize}
  \item[(1)] For every chord of a Gauss diagram $\m{G}$, we can associate a closed path along the curve $\mathscr{C}(\m{G})$.
  \item[(2)] For every two non-intersecting chords of a Gauss diagram $\m{G}$, we can associate two closed paths along the curve $\mathscr{C}(\m{G})$ such that every chord crosses both of those chords correspondences to the point of intersection of the paths.
  \item[(3)] If a Gauss diagram $\m{G}$ is realizable (say by a plane curve $\mathscr{C}(\m{G})$), then for every closed path (say) $\mathscr{P}$ along $\mathscr{C}(\m{G})$ we can associate a coloring another part of $\mathscr{C}(\m{G})$ into two colors (roughly speaking we get ``inner'' and ``outer'' sides of $\mathscr{P}$ \textit{cf.} Jordan curve Theorem). If a Gauss diagram is not realizable then (\cite[Theorem 1.A]{GPV}) it defines a virtual plane curve $\mathscr{C}(\m{G})$. We shall show that there exists a closed path along $\mathscr{C}(\m{G})$ for which we cannot associate a well-defined coloring of $\mathscr{C}(\m{G})$, \textit{i.e.,} $\mathscr{C}(\m{G})$ contains a path is colored into two colors.
  \end{itemize}

Using these ideas we solve the problem of which Gauss diagram can be realized by knots.  We then give a matrix approach of realization of Gauss diagrams and then we present an algorithm to construct meanders.

\section{Preliminaries}
Recall that classically, a knot is defined as an embedding of the circle $S^1$ into $\mathbb{R}^3$, or equivalently into the $3$-sphere $S^3$, \textit{i.e}., a knot is a closed curve embedded on $\mathbb{R}^3$ (or $S^3$) without intersecting itself, up to ambient isotopy.

The projection of a knot onto a $2$-manifold is considered with all multiple points are transversal double with will be call {\it crossing points} (or shortly \textit{crossings}). Such a projection is called the {\it shadow} by the knots theorists \cite{A,S2}, following \cite{S1} we shall also call these projections as plane curves. {\it A knot diagram} is a generic immersion of a circle $S^1$ to a plane $\mathbb{R}^2$ enhanced by information on overpasses and underpasses at double points.

\subsection{Gauss Diagrams} A generic immersion of a circle to a plane is characterized by its Gauss diagram \cite{PV}.

\begin{figure}[h!]
  \begin{center}
\begin{tikzpicture}[scale = 4]
 \draw [line width = 2, name path= a] (0.1,0.8) to [out = 0, in = 180] (0.6,0.2);
 \draw[line width =2, name path= b](0.6,0.2) to [out = 0, in = 270] (1.1,0.6);
 \draw[line width =2, name path= c] (1.1,0.6)to [out = 90, in =0] (0.6, 1);
 \draw[line width =2, name path= d](0.6,1) to [out = 180, in = 90] (0.3,0.4);
 \draw[line width =2, name path= e] (0.3,0.4)to [out= 270, in = 180] (0.7,-0.2);
 \draw[line width =2, name path= f](0.7,-0.2)to [out= 0, in = 270] (1.1,0.2);
 \draw[line width =2, name path= g](1.1,0.2)to [out = 90, in =270] (0.5, 1);
 \draw[line width =2, name path= h] (0.5,1)to [out = 90, in = 180] (0.7, 1.3);
 \draw[line width =2, name path= k](0.7,1.3) to [out= 0, in = 90] (1,0.9);
 \draw[line width =2, name path= l](1,0.9)to [out= 270, in = 60] (0.3, 0);
 \draw[line width =2, name path= m](0.3,0) to [out = 240, in = 180] (0.1,0.8);
 \fill [name intersections={of=a and d, by={1}}]
(1) circle (1pt) node[left] {$1$};
 \fill [name intersections={of=a and l, by={2}}]
(2) circle (1pt) node[above =1mm of 2] {$2$};
 \fill [name intersections={of=b and g, by={3}}]
(3) circle (1pt) node[right = 1mm of 3] {$3$};
 \fill [name intersections={of=c and l, by={4}}]
(4) circle (1pt) node[above right] {$4$};
 \fill [name intersections={of=d and g, by={5}}]
(5) circle (1pt) node[above left] {$5$};
 \fill [name intersections={of=e and l, by={6}}]
(6) circle (1pt) node[left = 1mm of 6] {$6$};
 \fill [name intersections={of=g and l, by={7}}]
(7) circle (1pt) node[above = 1mm of 7] {$7$};
 \begin{scope}[scale = 0.3, xshift = 7cm, yshift = 1.6cm, line width=2]
   \draw (0,0) circle (2);
    \coordinate (1) at (90:2);
    \node at (1) [above =1mm of 1] {$1$};
    \fill (1) circle(2pt);
    \coordinate (2) at (115:2);
    \node at (2) [above =1mm of 2] {$2$};
    \fill (2) circle(2pt);
    \coordinate (3) at (140:2);
    \node at (3) [above left = -1mm of 3] {$3$};
    \fill (3) circle(2pt);
    \coordinate (4) at (165:2);
    \node at (4) [left] {$4$};
    \fill (4) circle(2pt);
    \coordinate (5) at (190:2);
    \node at (5) [left] {$5$};
    \fill (5) circle(2pt);
    \coordinate (6) at (215:2);
    \node at (6) [left] {$1$};
    \fill (6) circle(2pt);
    \coordinate (7) at (240:2);
    \node at (7) [below] {$6$};
    \fill (7) circle(2pt);
    \coordinate (8) at (265:2);
    \node at (8) [below] {$3$};
    \fill (8) circle(2pt);
    \coordinate (9) at (290:2);
    \node at (9) [below] {$7$};
    \fill (9) circle(2pt);
    \coordinate (10) at (315:2);
    \node at (10) [below right] {$5$};
    \fill (10) circle(2pt);
    \coordinate (11) at (340:2);
    \node at (11) [below right] {$4$};
    \fill (11) circle(2pt);
    \coordinate (12) at (5:2);
    \node at (12) [right] {$7$};
    \fill (12) circle(2pt);
    \coordinate (13) at (30:2);
    \node at (13) [right] {$2$};
    \fill (13) circle(2pt);
    \coordinate (14) at (60:2);
    \node at (14) [above right] {$6$};
    \fill (14) circle(2pt);
    \draw[line width =2] (1) -- (6);
    \draw[line width =2] (3) -- (8);
    \draw[line width =2] (5) -- (10);
    \draw[line width =2] (4) -- (11);
    \draw[line width =2] (9) -- (12);
    \draw[line width =2] (2) -- (13);
    \draw[line width =2] (7) -- (14);
   \end{scope}
\end{tikzpicture}
\end{center}
\caption{The plane curve and its Gauss diagram are shown.}\label{exofgauss}
\end{figure}
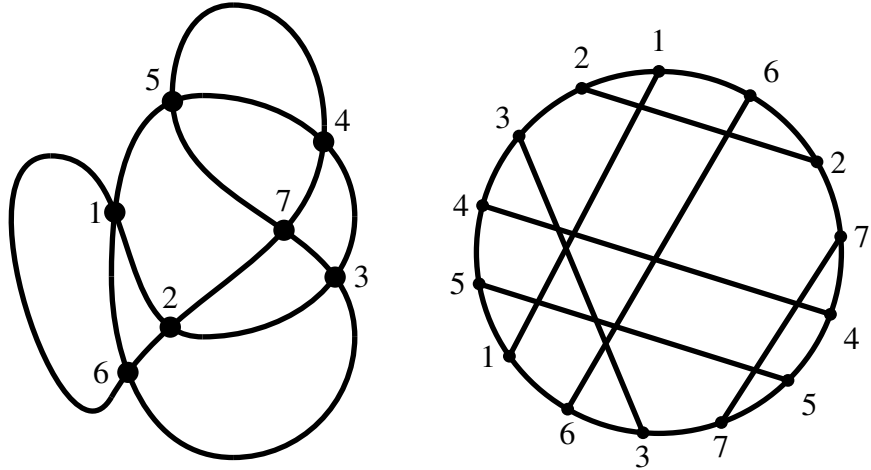

\begin{definition}
 {\it The Gauss diagram} is the immersing circle with the preimages of each double point connected with a chord.
\end{definition}

On the other words, this natation can be defined as follows. Let us walk on a path along the plane curve until returning back to the origin and then generate a word $W$ which is the sequence of the crossings in the order we meet them on the path. $W$ is a double occurrence word. If we put the labels of the crossing on a circle in the order of the word $W$ and if we join by a chord all pairs of identical labels then we obtain a chord diagram (=Gauss diagram) of the plane curve (see {\sc Figure} \ref{exofgauss}).

{\it A virtual knot diagram} \cite{GPV} is a generic immersion of the circle into the plane, with
double points divided into real crossing points and virtual crossing points, with the
real crossing points enhanced by information on overpasses and underpasses (as for
classical knot diagrams). At a virtual crossing the branches are not divided into an
overpass and an underpass. The Gauss diagram of a virtual knot is constructed in
the same way as for a classical knot, but all virtual crossings are disregarded.

\begin{theorem}{\cite[Theorem 1.A]{GPV}}\label{virt}
  A Gauss diagram defines a virtual knot diagram up to virtual moves.
\end{theorem}

Arguing similarly as in the real knot case, one can define {\it a shadow of the virtual knot} (see {\sc Figure} \ref{virt1}).

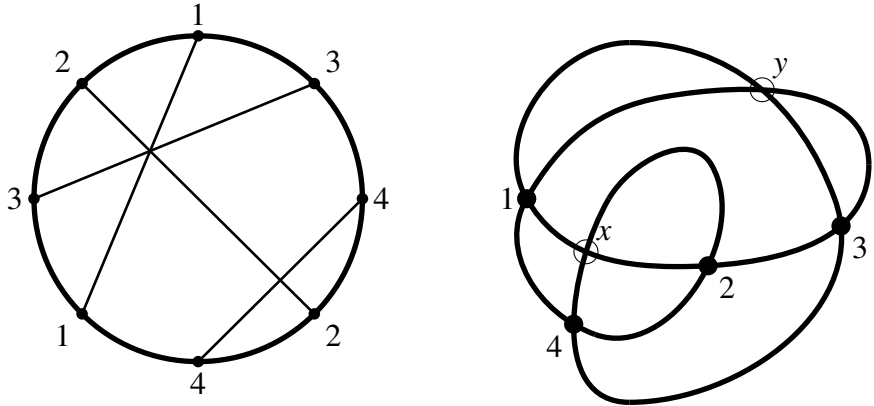
\begin{figure}[h!]
 \begin{tikzpicture}[scale =0.9]
  \draw [line width = 2, name path= a] (0,0) to [out = 300, in = 180] (2.2,-1) to [out = 0, in= 270] (5,0.5);
  \draw[line width =2,name path =b] (5,0.5) to [out = 90, in = 10] (2,1.5) to [out = 190, in = 60] (0,0);
  \draw[line width =2, name path =c1] (0,0) to [out = 240, in = 160] (1,-2);
  \draw[line width =2, name path =c2] (1,-2) to [out = 340, in = 300] (2.7, 0.5) to [out = 120, in=60] (1.2, 0);
  \draw[line width =2, name path = d] (1.2, 0) to [out = 240, in = 180] (1.5,-3);
  \draw[line width =2, name path =e] (1.5, -3) to [out = 0, in = 290] (4.5,0) to [out = 110, in =0] (1.5,2.3);
  \draw[line width =2] (1.5,2.3) to [out =180, in = 120] (0,0);
    \fill (0,0) circle (4pt) node[left] {$1$};
    \draw [name intersections={of=a and d, by={x}}]
(x) circle (5pt) node[above right] {$x$};
    \fill[name intersections = {of =a and c2, by ={2}}]
(2) circle (4pt) node[below right] {$2$};
    \fill[name intersections ={of =a and e, by ={3}}]
(3) circle (4pt) node[below right] {$3$};
    \draw [name intersections={of=b and e, by={y}}]
(y) circle (5pt) node[above right] {$y$};
\fill[name intersections ={of =c1 and d, by ={4}}]
(4) circle (4pt) node[below left] {$4$};
\begin{scope}[xshift = -4.8cm, scale =-0.8]
\draw[line width =2] (0,0) circle (3);
\draw[line width =1] (270:3)--(45:3);
\draw[line width =1] (315:3)--(135:3);
\draw[line width =1] (0:3)--(225:3);
\draw[line width =1] (90:3)--(180:3);
\fill(270:3) circle(3pt) node[above] {$1$};
\fill(315:3) circle(3pt) node[above left] {$2$};
\fill(0:3) circle(3pt) node[left] {$3$};
\fill(45:3) circle(3pt) node[below left] {$1$};
\fill(90:3) circle(3pt) node[below] {$4$};
\fill(135:3) circle(3pt) node[below right] {$2$};
\fill(180:3) circle(3pt) node[right] {$4$};
\fill(225:3) circle(3pt) node[above right] {$3$};
\end{scope}
\end{tikzpicture}
\caption{The chord diagram and the shadow of the virtual knot are shown. Here $x$ and $y$ are the virtual crossing points.}\label{virt1}
\end{figure}

\subsection{Conway's Smoothing}
We frequently use the following notations. Let $K$ be a knot, $\mathscr{C}$ its shadow and $\m{G}$ the Gauss diagram of $\mathscr{C}$. For every crossing $c$ of $\mathscr{C}$ we denote by $\m{c}$ the corresponding chord of $\m{G}$.

If a Gauss diagram $\m{G}$ contains a chord $\m{c}$ then we write $\m{c}\in\m{G}$. We denote by $\m{c}_0$, $\m{c}_1$ the endpoints of every chord $\m{c}\in \m{G}$. We shall also consider every chord $\m{c} \in \m{G}$ together with one of two arcs are between its endpoints, and a chosen arc is denoted by $\m{c}_0\m{c}_1$.

Further, $\m{c}_\times$ denotes the set of all chords cross the chord $\m{c}$ and $\m{c}_\parallel$ denotes the set of all chords do not cross the chord $\m{c}$. We put $\m{c} \not \in \m{c}_\times$, and $\m{c} \in \m{c}_\parallel$.

\textbf{Throughout this paper we consider Gauss diagrams such that $\m{c}_\times \ne \varnothing$ for every $\m{c \in G}$.}

As well known, John Conway introduced a ``surgical'' operation on knots, called \textit{smoothing}, consists in eliminating the crossing by interchanging the strands ({\sc Figure} \ref{Conway}).

\begin{figure}[h!]
 \begin{tikzpicture}[scale =1]
    \draw[->, line width =3] (1,0) -- (0,1);
    \draw[line width = 9,white] (0,0) -- (1,1);
    \draw[->, line width = 3] (0,0) -- (1,1);
    \draw[->,line width = 2] (1.3,0.5) -- (2.3,0.5);
    \draw[line width =3] (2.6, 0) to [out = 30, in = 270] (3,0.5);
    \draw[->,line width =3] (3,0.5) to [out =90,in = 330] (2.6,1);
    \draw[line width =3] (3.6,0) to [out = 150, in = 270] (3.2,0.5);
    \draw[->,line width =3] (3.2, 0.5) to [out = 90, in = 210] (3.6,1);
   \draw[->,line width =3] (0,-2)--(1,-1);
   \draw[->, line width =9,white] (1,-2) -- (0,-1);
   \draw[->, line width =3] (0,-1) -- (1,-2);
   \draw[->,line width = 2] (1.3,-1.5) -- (2.3,-1.5);
   \draw[line width=3] (2.6,-1) to [out = 300, in = 180] (3.1,-1.4);
   \draw[->,line width =3] (3.1, -1.4) to [out = 0, in = 240](3.6,-1);
   \draw[line width =3] (2.6, -2) to [out = 60, in =180](3.1,-1.6);
   \draw[->,line width =3] (3.1,-1.6) to [out = 0, in = 130] (3.6,-2);
  \end{tikzpicture}
    \caption{The Conway smoothing the crossings are shown.}\label{Conway}
\end{figure}
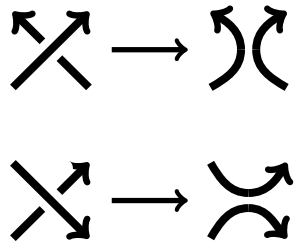

We aim to specialize a Conway smoothing a crossing of a plane curve to an operation on chords of the corresponding Gauss diagram.

Let $K$ be a knot, $\mathscr{C}$ its shadow, and $\m{G}$ the Gauss diagram of $\mathscr{C}$. Take a crossing point $c$ of $\mathscr{C}$ and let $D_c$ be a small disk centered at $c$ such that $D_c\cap \mathscr{C}$ does not contain another crossings of $\mathscr{C}$. Denote by $\partial D_c$ the boundary of $D_c$. Starting from $c$, let us walk on a path along the curve $\mathscr{C}$ until returning back to $c$. Denote this path by $\mathscr{L}_c$ and let $c c_a^lc_z^lc$ be the sequence of the points in the order we meet them on $\mathscr{L}_c$, where $ \{c_a^l,c_z^l\} = \mathscr{L}_c \cap \partial D_c$. After returning back to $c$ let us keep walking along the curve $\mathscr{C}$ in the same direction as before until returning back to $c$. Denote the corresponding path by $\mathscr{R}_c$ and let $cc_a^rc_z^rc$ be the sequence of the points in the order we meet them on $\mathscr{R}_c$, where $\{c_a^r, c_z^r\} = \partial D_c \cap \mathscr{R}_c$.

\begin{figure}[h!]
 \begin{center}
  \begin{tikzpicture}[scale=0.8]
   \fill[lightgray!20] (0.5,-0.5) circle(15pt);
   \draw [thick, name path= a] (0.5,-0.5) to [out = 150, in = 60] (-3,1);
   \draw [thick, name path= a1](-3,1) to [out = 240, in =190] (-2,-0.5);
   \draw [thick, name path= a2] (-2,-0.5) to [out=10, in = 270] (-0.5,1);
   \draw [thick, name path= a3](-0.5,1) to [out = 90, in = 0] (-1,2);
   \draw [thick, name path= a4](-1,2) to [out = 180, in = 100]
   (-2.2,-0.6) to [out = 280, in = 240] (0.5,-0.5);
   \draw[line width=2, name path =b] (0.5,-0.5) to [out = 60, in = 150] (2,0);
   \draw[line width=2, name path =b1] (2,0) to [out = 330, in = 320] (0,-1.8) to
   [out = 140, in=300] (-1,-0.5) to [out = 120, in = 250] (-0.7, 2.2) to [out=70, in = 80]
   (-4,1) to [out =260, in = 180] (-1,-2);
   \draw [line width = 2, name path= b2](-1,-2) to [out=0, in =330] (0.5,-0.5);
   \fill (0.5,-0.5) circle (4pt);
   \draw[name path = circ] (0.5,-0.5) circle (15pt);
   \draw (0.5,-0.5) node [right] {\small $c$};
   \fill [name intersections={of=circ and a, by={xla}}]
   (xla) circle (2pt) node[above,black] {\tiny$c^l_a$};
   \fill [name intersections={of=circ and a4, by={xlz}}]
   (xlz) circle (2pt) node[below,black] {\tiny $c^l_z$};
   \fill [name intersections={of=circ and b, by={xra}}]
   (xra) circle (2pt) node[above,black] {\tiny $c^r_a$};
   \fill [name intersections={of=circ and b2, by={xrz}}]
   (xrz) circle (2pt) node[below] {\tiny $c^r_z$};
   \fill [name intersections={of=a and a2, by={1}}]
   (1) circle (3pt) node[below =0.5mm of 1,black] {$1$};
   \fill [name intersections={of= a and b1, by={2}}]
   (2) circle (3pt) node[below left = -1mm of 2] {$2$};
   \fill [name intersections={of= a and a4, by={3}}]
   (3) circle (3pt) node[above] {$3$};
   \fill [name intersections={of= a1 and a4, by={4}}]
   (4) circle (3pt) node[below left = -1mm of 4] {$4$};
   \fill [name intersections={of= a2 and b1, by={5}}]
   (5) circle (3pt) node[below = 1mm of 5] {$5$};
   \fill [name intersections={of= a3 and b1, by={6}}]
   (6) circle (3pt) node[right ] {$6$};
   \fill [name intersections={of= a4 and b1, by={7}}]
   (7) circle (3pt) node[below] {$7$};
   \fill [name intersections={of= b1 and b2, by={8}}]
   (8) circle (3pt) node[below] {$8$};
 \begin{scope}[xshift=7cm]
   \draw [thick, name path= a] (0.5,-0.5) to [out = 150, in = 60] (-3,1);
   \draw [thick, name path= a1](-3,1) to [out = 240, in =190] (-2,-0.5);
   \draw [thick, name path= a2] (-2,-0.5) to [out=10, in = 270] (-0.5,1);
   \draw [thick, name path= a3](-0.5,1) to [out = 90, in = 0] (-1,2);
   \draw [thick, name path= a4](-1,2) to [out = 180, in = 100]
   (-2.2,-0.6) to [out = 280, in = 240] (0.5,-0.5);
   \draw[line width=2, name path =b] (0.5,-0.5) to [out = 60, in = 150] (2,0);
   \draw[line width=2, name path =b1] (2,0) to [out = 330, in = 320] (0,-1.8) to
[out = 140, in=300] (-1,-0.5) to [out = 120, in = 250] (-0.7, 2.2) to [out=70, in = 80]
(-4,1) to [out =260, in = 180] (-1,-2);
   \draw [line width = 2, name path= b2](-1,-2) to [out=0, in =330] (0.5,-0.5);
   \fill (0.5,-0.5) circle (4pt);
   \fill[name path = circ, white] (0.5,-0.5) circle (15pt);
   \fill [name intersections={of=circ and a, by={xla}}]
   (xla) circle (2pt) node[above,black] {\tiny$c^l_a$};
   \fill [name intersections={of=circ and a4, by={xlz}}]
   (xlz) circle (2pt) node[above =-0.5mm of xlz] {\tiny $c^l_z$};
   \fill [name intersections={of=circ and b, by={xra}}]
   (xra) circle (2pt) node[above] {\tiny $c^r_a$};
   \fill [name intersections={of=circ and b2, by={xrz}}]
   (xrz) circle (2pt) node[above = -0.5mm of xrz] {\tiny $c^r_z$};
   \fill [name intersections={of=a and a2, by={1}}]
   (1) circle (3pt) node[below =0.5mm of 1,black] {$1$};
   \fill [name intersections={of= a and b1, by={2}}]
   (2) circle (3pt) node[below left = -1mm of 2] {$2$};
   \fill [name intersections={of= a and a4, by={3}}]
   (3) circle (3pt) node[above] {$3$};
   \fill [name intersections={of= a1 and a4, by={4}}]
   (4) circle (3pt) node[below left = -1mm of 4] {$4$};
   \fill [name intersections={of= a2 and b1, by={5}}]
   (5) circle (3pt) node[below = 1mm of 5] {$5$};
   \fill [name intersections={of= a3 and b1, by={6}}]
   (6) circle (3pt) node[right ] {$6$};
   \fill [name intersections={of= a4 and b1, by={7}}]
   (7) circle (3pt) node[below] {$7$};
   \fill [name intersections={of= b1 and b2, by={8}}]
   (8) circle (3pt) node[below] {$8$};
   \draw (xrz) to  (xlz);
   \draw (xla) to (xra);
 \end{scope}
 \begin{scope} [yshift = -7cm, xshift = -2cm, line width=2]
   \draw (0,0) circle (3);
   \draw[thick] (120:3) -- (20:3);
   \draw[thick] (60:3) -- (160:3);
   \draw[thick] (80:3) -- (180:3);
   \draw[line width=4] (0:3) -- (220:3);
   \draw[dotted] (200:3)--(260:3);
   \draw[dotted] (100:3)--(280:3);
   \draw[dotted] (40:3)--(300:3);
   \draw[dotted] (140:3)--(320:3);
   \draw[line width =2] (240:3) -- (340:3);
   \fill (0:3) circle (4pt) node[right] {$\m{c}$};
   \fill (20:3) circle (3pt) node[right] {$1$};
   \fill (40:3) circle (3pt) node[right] {$2$};
   \fill (60:3) circle (3pt) node[above] {$3$};
   \fill (80:3) circle (3pt) node[above,black] {$4$};
   \fill (100:3) circle (3pt) node[above,black] {$5$};
   \fill (120:3) circle (3pt) node[above] {$1$};
   \fill (140:3) circle (3pt) node[above] {$6$};
   \fill (160:3) circle (3pt) node[left] {$3$};
   \fill (180:3) circle (3pt) node[left] {$4$};
   \fill (200:3) circle (3pt) node[left] {$7$};
   \fill (220:3) circle (4pt) node[left] {$\m{c}$};
   \fill (240:3) circle (3pt) node[below] {$8$};
   \fill (260:3) circle (3pt) node[below] {$7$};
   \fill (280:3) circle (3pt) node[below] {$5$};
   \fill (300:3) circle (3pt) node[below] {$2$};
   \fill (320:3) circle (3pt) node[below] {$6$};
   \fill (340:3) circle (3pt) node[right] {$8$};
   \end{scope}
   \begin{scope} [yshift = -7cm, xshift = 7cm, line width=2]
   \draw (0,0) circle (3);
   \draw[thick] (120:3) -- (20:3);
   \draw[thick] (60:3) -- (160:3);
   \draw[thick] (80:3) -- (180:3);
   \draw[dotted] (200:3)--(320:3);
   \draw[dotted] (100:3)--(300:3);
   \draw[dotted] (40:3)--(280:3);
   \draw[dotted] (140:3)--(260:3);
   \draw[line width =2] (240:3) -- (340:3);
   \fill (20:3) circle (3pt) node[right] {$1$};
   \fill (40:3) circle (3pt) node[right] {$2$};
   \fill (60:3) circle (3pt) node[above] {$3$};
   \fill (80:3) circle (3pt) node[above,black] {$4$};
   \fill (100:3) circle (3pt) node[above,black] {$5$};
   \fill (120:3) circle (3pt) node[above] {$1$};
   \fill (140:3) circle (3pt) node[above] {$6$};
   \fill (160:3) circle (3pt) node[left] {$3$};
   \fill (180:3) circle (3pt) node[left] {$4$};
   \fill (200:3) circle (3pt) node[left] {$7$};
   \fill (240:3) circle (3pt) node[below] {$8$};
   \fill (260:3) circle (3pt) node[below] {$6$};
   \fill (280:3) circle (3pt) node[below] {$2$};
   \fill (300:3) circle (3pt) node[below] {$5$};
   \fill (320:3) circle (3pt) node[below] {$7$};
   \fill (340:3) circle (3pt) node[right] {$8$};
   \end{scope}
\end{tikzpicture}
\end{center}
\caption{The Conway smoothing the crossing $c$ and the chord $\m{c}$ are shown.}\label{del}
\end{figure}
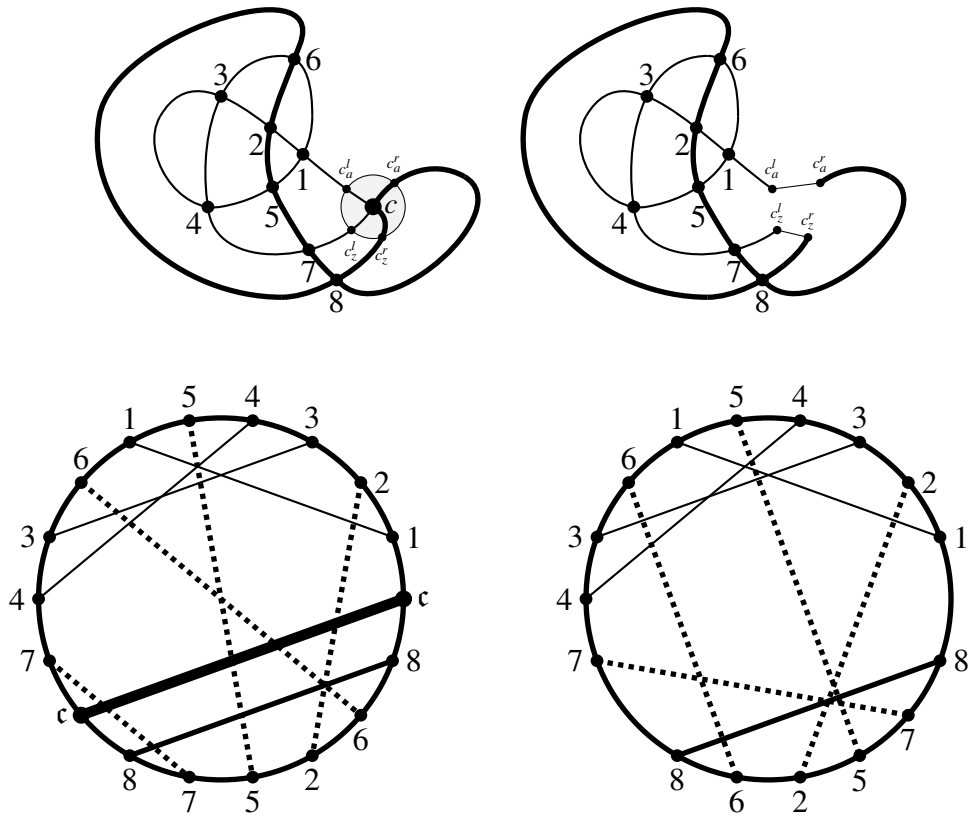

Next, let us delete the inner side of $D_c \cap \mathscr{C}$ and attach $c_l^a$ to $c_r^a$, and $c^r_z$ to $c^l_z$. We thus get the new plane curve $\widehat{\mathscr{C}}_c$ (see {\sc Figure} \ref{del}). It is easy to see that this curve is the shadow of the knot, which is obtained from $K$ by Conway's smoothing the crossing $c$. Let $\widehat{\m{G}}_\m{c}$ be the Gauss diagram of $\widehat{\mathscr{C}}_c$. We shall say that \textit{the Gauss diagram $\widehat{\m{G}}_\m{c}$ is obtained from the Gauss diagram $\m{G}$ by Conway's smoothing the chord $\m{c}$.}

As an immediate consequence of the preceding discussion, we get the following proposition.

\begin{proposition}\label{RemConway}
 Let $\m{G}$ be a Gauss diagram and $\m{c}$ be its arbitrary chord. Then $\widehat{\m{G}}_{\m{c}}$ is obtained from $\m{G}$ as follows: (1) delete the chord $\m{c}$, (2) if two chords $\m{a}, \m{b} \in \m{c}_\times$ intersect (\textit{resp.} do not intersected) in $\m{G}$ then they do not intersect in $\m{\widehat{G}_c}$ (\textit{resp.} intersected), (3) another chords keep their positions.
\end{proposition}
\begin{proof}
  Indeed, let $W$ be the word which is the sequence of the crossings in the order we meet them on the curve $\mathscr{C}$. Since $\m{c}_\times \ne \varnothing$, $W$ can be written as follows $W = W_1 c W_2 c W_3$, where $W_1,W_2,W_3$ are subwords of $W$ and at least one of $W_1,W_3$ is not empty. Define $W_2^R$ as the reversal of the word $W_2$. Then, from the preceding discussion, the word $\widehat{W}_c: = W_1 W_2^R W_3$ gives $\widehat{\m{G}}_{\m{c}}$ (see {\sc Figure} \ref{del}) and the statement follows.
\end{proof}

\section{Partitions of Gauss Diagrams}
In this section we introduce notations, whose importance will become clear as we proceed.

\begin{definition}
  Let $\m{G}$ be a Gauss diagram and $\m{a}$ a chord of $\m{G}$. \textit{A $C$-contour}, denoted $C(\m{a})$, consists of the chord $\m{a}$, a chosen arc $\m{a}_0\m{a}_1$, and all chords of $\m{G}$ such that all their endpoints lie on the arc $\m{a}_0\m{a}_1$. We call a chord from the set $\m{a}_\times$ \textit{the door chord of the $C$-contour $C(\m{a})$}.
\end{definition}

Let us consider a plane curve $\mathscr{C}:S^1 \to \mathbb{R}^2$ and let $\m{G}$ be its Gauss diagram. Every chord $\m{c \in G}$ correspondences to the crossing $c$ of $\mathscr{C}$. Thus for every $C$-contour $C(\m{c})$, we can associate a closed path $\mathscr{C}(c)$ along the curve $\mathscr{C}$. We call $\mathscr{C}(c)$ \textit{ the loop of the curve} $\mathscr{C}$. It is obviously that there is the one-to-one correspondence between self-intersection points of $\mathscr{C}(c)$ and all chords from $C(\m{c})$.

\begin{figure}[h!]
 \begin{center}
 \begin{tikzpicture}[scale = 4.5]
 \draw [line width = 3, name path= a] (0.3,0.4) to [out = 300, in = 180] (0.6,0.2) to [out = 0, in= 270] (1,0.4);
 \draw [line width = 3, name path= aa] (1,0.4) to [out = 90, in = 0] (0.6, 0.8) to [out = 180, in = 75] (0.3,0.4);
 \draw [thick, white] (0.3,0.4) to [out = 300, in = 180] (0.6,0.2) to [out = 0, in= 270] (1,0.4);
 \draw [thick, white] (1,0.4) to [out = 90, in = 0] (0.6, 0.8) to [out = 180, in = 75] (0.3,0.4);
 \draw [line width =3] (0.3,0.4) [out = 255, in = 140] to (0.4, 0);
 \draw[line width =3, name path= b] (0.4,0) to [out = 320, in =180] (0.6,-0.1);
 \draw[line width =3, name path= c](0.6,-0.1) to [out= 0, in = 270] (0.8,0.2);
 \draw[line width =3, name path= d] (0.8,0.2)to [out = 90, in =270] (0.6, 0.9) to [out = 90, in = 180] (0.7, 1) to [out= 0, in = 90] (0.9,0.8);
 \draw[line width =3, name path= bb](0.9,0.8)to [out= 270, in = 60] (0.4, 0);
 \draw[line width =2, dashed, white] (0.4,0) to [out = 320, in =180] (0.6,-0.1);
 \draw[line width =2, dashed, white](0.6,-0.1) to [out= 0, in = 270] (0.8,0.2);
 \draw[line width =2, dashed, white] (0.8,0.2)to [out = 90, in =270] (0.6, 0.9) to [out = 90, in = 180] (0.7, 1) to [out= 0, in = 90] (0.9,0.8);
 \draw[line width =2, dashed, white](0.9,0.8)to [out= 270, in = 60] (0.4, 0);
 \draw [line width =3] (0.4,0) to [out = 240, in =0] (0.3, -0.1) to [out = 180, in =150] (0.3,0.4);
  \fill [name intersections={of=a and bb, by={1}}]
(1) circle (0.7pt) node[above] {$1$};
\fill [name intersections={of=a and d, by={2}}]
(2) circle (0.7pt) node[above left] {$2$};
\fill [name intersections={of=aa and bb, by={3}}]
(3) circle (0.7pt) node[above right] {$3$};
\fill [name intersections={of=aa and d, by={4}}]
(4) circle (0.7pt) node[above left] {$4$};
\fill [name intersections={of=bb and d, by={5}}]
(5) circle (0.7pt) node[right] {$6$};
\fill (0.3,0.4) circle(1pt) node[above left] {$0$};
\fill (0.4,0) circle(1pt) node[below =1mm] {$5$};
\begin{scope}[scale = 0.25, xshift = 8cm, yshift = 2cm, line width=2]
   \draw[line width =3] (0,0) circle (2);
   \begin{scope}[rotate =180]
   \draw[white, thick] (2,0) arc(0:125:2);
   \end{scope}
    \begin{scope}[rotate =330]
   \draw[dashed,  white] (2,0) arc(0:185:2);
   \end{scope}
    \draw[black] (180:2) -- (305:2);
    \draw[thick, white] (180:2) -- (305:2);
    \draw (230:2)--(355:2);
    \draw (280:2)--(55:2);
    \draw (255:2)--(80:2);
    \draw[line width=3] (20:2)--(105:2);
    \draw[white, dashed,line width=3] (20:2)--(105:2);
    \draw (205:2)--(130:2);
    \draw[line width=4] (330:2)--(155:2);
    \draw[white, dashed,line width=4] (330:2)--(155:2);
    \fill (180:2) circle(3pt) node[left] {$0$};
    \fill (205:2) circle(3pt) node[left] {$1$};
    \fill (230:2) circle(3pt) node[below] {$2$};
    \fill (255:2) circle(3pt) node[below] {$3$};
    \fill (280:2) circle(3pt) node[below] {$4$};
    \fill (305:2) circle(3pt) node[below] {$0$};
    \fill (305:2) circle(3pt) node[below] {$0$};
    \fill (330:2) circle(3pt) node[right] {$5$};
    \fill (355:2) circle(3pt) node[right] {$2$};
    \fill (20:2) circle(3pt) node[right] {$6$};
    \fill (55:2) circle(3pt) node[above] {$4$};
    \fill (80:2) circle(3pt) node[above] {$3$};
    \fill (105:2) circle(3pt) node[above] {$6$};
    \fill (130:2) circle(3pt) node[above] {$1$};
    \fill (155:2) circle(3pt) node[left] {$5$};
  \end{scope}
 \end{tikzpicture}
 \end{center}
  \caption{Every $C$-contour of the Gauss diagram correspondences to the closed path along the plane curve and vise versa. We see that the chord $6$ correspondences to the self-intersection point $6$ of the dotted loop.}\label{arc}
\end{figure}
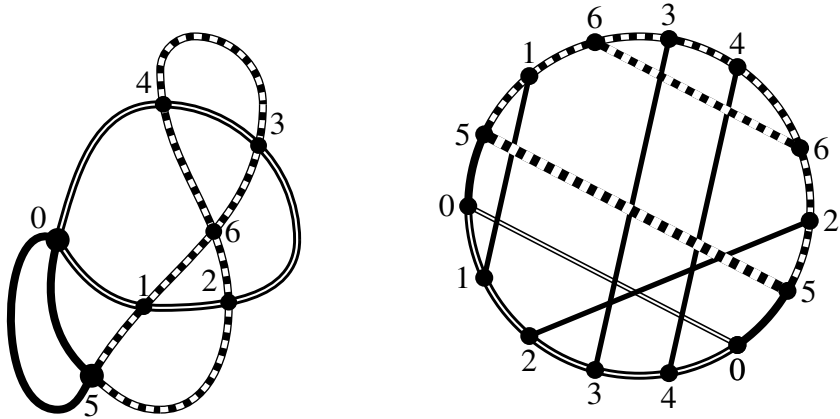

\begin{example}
  In {\sc Figure} \ref{arc} the plane curve $\mathscr{C}$ and its Gauss diagram are shown. Consider the (cyan) $C$-contour $C(5)$. We see that $\mathscr{C}(5)$ is the closed path along the curve. It is the self-intersecting path and we see that the crossing $6$ correspondences to the chord from the set $5_\parallel$. Further, the red closed path $\mathscr{C}(0)$ correspondences to the red $C$-contour $C(0)$.
\end{example}

\begin{definition}
Let $\m{G}$ be a Gauss diagram, $\m{a},\m{b}$ its intersecting chords. \textit{An $X$-contour}, denoted $X(\m{a},\m{b})$, consists of two non-intersecting arcs $\m{a}_0\m{b}_0$, $\m{a}_1\m{b}_1$ and all chords of $\m{G}$ such that all their endpoints lie on $\m{a}_0\m{b}_0$ or on $\m{a}_1\m{b}_1$. A chord is called \textit{the door chord of the $X$-contour $X(\m{a},\m{b})$} if only one of its endpoints belongs to $X(\m{a},\m{b})$. We say that the $X$-contour $X(\m{a},\m{b})$ is \textit{non-degenerate} if it has at least one door chord, and it does not contain all chords of $\m{G}$.
\end{definition}

\begin{example}
Let us consider the Gauss diagram in {\sc Figure \ref{X-cont}}. The black chords are the door chords of the black $X$-contour $X(1,3)$. We see that the door chords correspondence to ``entrances'' and ``exits'' of the black closed path along the curve. We also see that this $X$-contour is non-degenerate.
\end{example}

The previous Example implies a partition of a Gauss diagram (\textit{resp.} a plane curve) into two parts.

\begin{definition}[\textbf{An $X$-contour coloring}]\label{coloring}
Given a Gauss diagram $\m{G}$ and its an $X$-contour. Let us walk along the circle of $\m{G}$ in a chosen direction and color all arcs of $\m{G}$ until returning back to the origin as follows: (1) we don't colors the arcs of the $X$-contour, (2) we use only two different colors, (3) we change a color whenever we meet an endpoint of a door chord.
\end{definition}

Similarly, one can define \textbf{a $C$-contour coloring} of a Gauss diagram $\m{G}$.

\begin{remark}
Let $\m{G}$ be a Gauss diagram and $\mathscr{C}$ the corresponding (may be virtual) plane curve, \textit{i.e.,} $\m{G}$ determines the curve $\mathscr{C}$. For every $X$-contour $X(\m{a},\m{b})$ in $\m{G}$, we can associate the closed path along the curve $\mathscr{C}$. We call this path \textit{the $\mathscr{X}$-contour} and denote by $\mathscr{X}(a,b)$. Similarly one can define \textit{door crossing} for $\mathscr{X}(a,b)$.

Further, for the $X(\m{a},\m{b})$-contour coloring of $\m{G}$, we can associate \textit{$\mathscr{X}(a,b)$-contour coloring of the curve $\mathscr{C}$}.

Next, let $\m{G}$ be a realizable Gauss diagram determines the plane curve $\mathscr{C}$ and let $X(\m{a},\m{b})$ be an $X$-contour of $\m{G}$ such that $\mathscr{X}(a,b)$ is the non-self-intersecting path (= the Jordan curve). Then the $\mathscr{X}(a,b)$-contour coloring of $\mathscr{C}$ divides the curve $\mathscr{C}$ into two colored parts, \textit{cf.} Jordan curve Theorem.
\end{remark}

\begin{figure}[h!]
 \begin{tikzpicture}
 \begin{scope}[xshift = 3cm,scale =0.9]
  \draw [line width = 3,name path =v1] (-1,0) to [out = 30, in = 180] (1,-0.2);
  \draw [line width = 2,white] (-1,0) to [out = 30, in = 180] (1,-0.2);
  \draw[line width =3, name path= o1] (1,-0.2) to [out =0, in = 200] (4,0.2);
  \draw [line width = 3, name path= a] (4,0.2) to [out =20 , in = 270] (5,1.5);
  \draw [line width = 2, white] (4,0.2) to [out =20 , in = 270] (5,1.5);
  \draw [line width = 3, name path= b] (5,1.5) to [out = 90, in = 70] (0,1);
  \draw [line width = 2, white] (5,1.5) to [out = 90, in = 70] (0,1);
  \draw [line width = 3, name path =c] (0,1) to [out = 250, in = 180] (0.3,-1);
  \draw [line width = 2, white] (0,1) to [out = 250, in = 180] (0.3,-1);
  \draw [line width = 3 ] (0.3,-1) to [out = 0, in = 250] (1,-0.2);
  \draw [line width = 2,white ] (0.3,-1) to [out = 0, in = 250] (1,-0.2);
  \draw [line width = 3, name path= o2] (1,-0.2) to [out = 70, in=180] (2.5,1.5);
  \draw [line width = 3, name path= o3] (2.5,1.5) to [out = 0, in = 120] (4,0.2);
  \draw [line width =3] (4,0.2) to [out = 300, in = 0] (3,-1);
  \draw [line width =2, white] (4,0.2) to [out = 300, in = 0] (3,-1);
  \draw [line width = 3 ] (0.3,-1) to [out = 0, in = 250] (1,-0.2);
 \draw [line width = 2, white ] (0.3,-1) to [out = 0, in = 250] (1,-0.2);
  \draw [name path= w, white] (3,-1) to [out = 180, in = 240] (2.5,0.3);
  \fill [name intersections={of=o1 and w, by={T4}}]
 (T4) circle (1pt);
 \draw [name path= h2, white] (T4) to (3.2,2);
  \fill [name intersections = {of =o3 and h2, by ={T2}}]
 (T2) circle (1pt);
 \draw [name path= h3, white] (T4) to (4,1);
  \fill [name intersections = {of =o3 and h3, by ={T3}}]
 (T3) circle (1pt);
 \draw [name path= h4, white] (T4) to (1,1.5);
  \fill [name intersections = {of =o2 and h4, by ={T1}}]
 (T1) circle (1pt);
\draw [line width =3] (3,-1) to [out = 180, in = 260] (T4);
\draw [line width =2, white] (3,-1) to [out = 180, in = 260] (T4);
\draw [line width =3,name path=y1] (T4) to [out = 80, in = 220] (T2);
\draw [line width =2,dashed,white] (T4) to [out = 80, in = 220] (T2);
\draw [line width =3, name path = d] (T2) to [out = 40, in = 300] (5,3);
\draw [line width =2, white] (T2) to [out = 40, in = 300] (5,3);
\draw [line width =3, name path = d1] (5,3) to [out = 120, in = 140] (0,0.5);
\draw [line width =2, white] (5,3) to [out = 120, in = 140] (0,0.5);
\draw [line width =3,name path=v2] (0,0.5) to [out = 320, in = 150] (T1);
\draw [line width =2, white] (0,0.5) to [out = 320, in = 150] (T1);
\draw [line width =3, name path = y2] (T1) to [out = 330, in = 200] (T3);
\draw [line width =2, dashed, white] (T1) to [out = 330, in = 200] (T3);
\draw [line width =3, name path = d2] (T3) to [out = 20, in = 90] (6,0);
\draw [line width =2, white] (T3) to [out = 20, in = 90] (6,0);
\draw [line width =3] (6,0) to [out = 270, in = 0] (2,-2);
\draw [line width =2, white] (6,0) to [out = 270, in = 0] (2,-2);
\draw [line width =3] (2,-2) to [out = 180, in = 270] (-1.3,-0.5);
\draw [line width =2, white] (2,-2) to [out = 180, in = 270] (-1.3,-0.5);
\draw [line width =3] (-1.3, -0.5) to [out = 90, in = 210] (-1,0);
\draw [line width =2, white] (-1.3, -0.5) to [out = 90, in = 210] (-1,0);
\fill (T1) circle (3pt) node [above] {$7$};
\fill (T2) circle (3pt) node [above] {$8$};
\fill (T3) circle (3pt) node [above] {$9$};
\fill (T4) circle (3pt) node [below right] {$2$};
 \fill [name intersections = {of =y1 and y2, by ={I}}]
 (I) circle (3pt) node [above left] {$10$};
  \fill (1,-0.2) circle (3pt) node [below right] {$1$};
  \fill (4,0.2) circle (3pt) node [below right] {$3$};
 \fill [name intersections = {of =v1 and c, by ={a1}}]
 (a1) circle (4pt) node [below left] {$0$};
 \fill [name intersections = {of =a and d2, by ={a2}}]
 (a2) circle (4pt) node [above right] {$4$};
  \fill [name intersections = {of =b and d, by ={a3}}]
 (a3) circle (4pt) node [above] {$5$};
  \fill [name intersections = {of =d1 and c, by ={a4}}]
 (a4) circle (4pt) node [above right] {$6$};
\end{scope}
 \begin{scope}[xshift =-2cm, yshift = 1cm,scale=0.6]
   \draw[line width =2] (0,0) circle (4);
   \begin{scope}[rotate =30]
   \draw[line width=3] (4,0) arc(0:45:4);
   \end{scope}
    \begin{scope}[rotate =150]
   \draw[line width=3] (4,0) arc(0:60:4);
   \end{scope}
   \begin{scope}[rotate =75]
   \draw[line width=5] (4,0) arc(0:75:4);
   \draw[ white, line width=3] (4,0) arc(0:75:4);
   \end{scope}
   \begin{scope}[rotate =210]
   \draw[ line width=5] (4,0) arc(0:15:4);
   \draw[ line width=3, white] (4,0) arc(0:15:4);
   \end{scope}
   \begin{scope}[rotate =225]
   \draw[line width=5] (4,0) arc(0:30:4);
   \draw[line width=2, white,dashed] (4,0) arc(0:30:4);
   \end{scope}
   \begin{scope}[rotate =255]
   \draw[ line width=5] (4,0) arc(0:45:4);
   \draw[ line width=3,white] (4,0) arc(0:45:4);
   \end{scope}
   \begin{scope}[rotate =300]
   \draw[ line width=5] (4,0) arc(0:30:4);
   \draw[ line width=2, dashed, white] (4,0) arc(0:30:4);
   \end{scope}
   \begin{scope}[rotate =330]
   \draw[ line width=5] (4,0) arc(0:60:4);
   \draw[ line width=3, white] (4,0) arc(0:60:4);
   \end{scope}
   \draw[line width=3] (30:4) -- (150:4);
   \draw[line width=3] (75:4) -- (210:4);
   \draw[line width =2] (55:4)--(225:4);
   \draw[line width =2] (165:4) -- (300:4);
   \draw[line width =2] (180:4) -- (255:4);
   \draw[line width =2] (195:4) -- (330:4);
   \draw[line width =3] (240:4) -- (315:4);
   \draw[line width =1, dashed, white] (240:4) -- (315:4);
   \draw[line width =2] (105:4) -- (270:4);
   \draw[line width =1, white] (105:4) -- (270:4);
   \draw[line width =2] (120:4) -- (285:4);
   \draw[line width =1, white] (120:4) -- (285:4);
   \draw[line width =2] (90:4) -- (350:4);
   \draw[line width =1, white] (90:4) -- (350:4);
   \draw[line width =2] (135:4) -- (10:4);
   \draw[line width =1, white] (135:4) -- (10:4);
    \fill (30:4) circle(4pt) node[above right] {$\m{1}$};
    \fill (55:4) circle(4pt) node[above right] {$\m{2}$};
    \fill (75:4) circle(4pt) node[above right] {$\m{3}$};
    \fill (90:4) circle(4pt) node[above] {$4$};
    \fill (105:4) circle(4pt) node[above] {$5$};
    \fill (120:4) circle(4pt) node[above] {$6$};
    \fill (135:4) circle(4pt) node[above] {$0$};
    \fill (150:4) circle(4pt) node[left] {$1$};
    \fill (165:4) circle(4pt) node[left] {$7$};
    \fill (180:4) circle(4pt) node[left] {$8$};
    \fill (195:4) circle(4pt) node[left] {$9$};
    \fill (210:4) circle(4pt) node[left] {$3$};
    \fill (225:4) circle(4pt) node[left] {$2$};
    \fill (240:4) circle(4pt) node[below] {$10$};
    \fill (255:4) circle(4pt) node[below] {$8$};
    \fill (270:4) circle(4pt) node[below] {$5$};
    \fill (285:4) circle(4pt) node[below] {$6$};
    \fill (300:4) circle(4pt) node[below] {$7$};
    \fill (315:4) circle(4pt) node[below] {$10$};
    \fill (330:4) circle(4pt) node[right] {$9$};
    \fill (350:4) circle(4pt) node[right] {$4$};
    \fill (10:4) circle(4pt) node[right] {$0$};
    \end{scope}
\end{tikzpicture}
 \caption{For the $X(1,3)$-contour coloring of the Gauss diagram, we associate the plane curve coloring. We see that the $\mathscr{X}$-contour $\mathscr{X}(1,3)$ (= black loop) divides the plane curve into two parts.}\label{X-cont}
\end{figure}
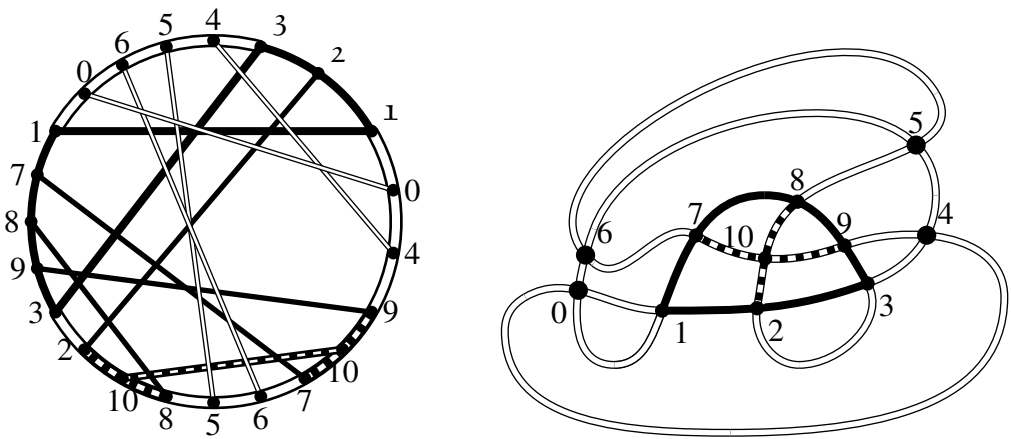

\section{The Even and The Sufficient Conditions}
If a Gauss diagram can be realized by a plane curve we then say that this Gauss diagram is \textit{realizable,} and \textit{non-realizable} otherwise. So, in this section, we give a criterion allowing verification and comprehension of whether a given Gauss diagram is realizable or not. Moreover, we give an explanation allowing comprehension of why the needed condition is not sufficient for realizability of Gauss diagrams.

\subsection{The Even Condition}

\begin{proposition}\label{prop1}
Let $\mathscr{C}:S^1 \to \mathbb{R}^2$ be a plane curve and $\m{G}$ its Gauss diagram. Then
\begin{itemize}
   \item[(1)] $|\m{a}_\times \cap \m{b}_\times| \equiv 0 \bmod 2$ for every two non-interesting chords $\m{a,b \in G}$,
  \item[(2)] $|\m{c}_\times| \equiv 0 \bmod 2$ for every chord $\m{c \in G}$.
\end{itemize}
\end{proposition}
\begin{proof}
Let $\m{a},\m{b}\in\m{G}$ be two non-intersecting chords of $\m{G}$. Take two $C$-contours $C(\m{a})$, $C(\m{b})$ such that their arcs $\m{a_0a_1}$, $\m{b_0b_1}$ do not intersect. It is obvious that for the loops $\mathscr{C}(a)$, $\mathscr{C}(b)$, we can associate the one-to-one correspondence between the set $\mathscr{C}(a) \cap \mathscr{C}(b)$ and the set $\m{a}_\times\cap\m{b}_\times$.

Because, by Proposition \ref{RemConway}, all chord from the set $\m{a}_\times \cap \m{b}_\times$ keep their positions in Gauss diagram $\m{\widehat{G}_c}$ (= Conway's smoothing the chord $\m{c}$) for every $\m{c} \in \m{a}_\parallel \cap \m{b}_\parallel \setminus \{\m{a},\m{b}\}$, it is sufficient to prove the statement in the case $\m{a}_\parallel \cap \m{b}_\parallel = \{\m{a},\m{b}\}$, \textit{i.e.,} the loops $\mathscr{C}(a)$, $\mathscr{C}(b)$ are non-self-intersecting loops (= the Jordan curves).

From Jordan curve Theorem, it follows that the loop $\mathscr{C}(a)$ divides the curve $\mathscr{C}$ into two regions, say, $\mathscr{I}$ and $\mathscr{O}$. Assume that $b\in \mathscr{O}$ and let us walk along the loop $\mathscr{C}(b)$. We say that an intersection point $c \in \mathscr{C}(a) \cap \mathscr{C}(b)$ is the entrance (\textit{resp.} the exit) if we shall be in the region $\mathscr{I}$ (\textit{resp.} $\mathscr{O}$) after meeting $c$ with respect to our walk. Since a number of entrances has to be equal to the number of exits, then $|\m{a}_\times \cap \m{b}_\times| \equiv 0\bmod 2$. Arguing similarly, we prove $|\m{c}_\times| \equiv 0 \bmod 2$ for every chord $\m{c}$.
\end{proof}

As an immediate consequence of Proposition \ref{prop1} we get the following.

\begin{corollary}[{\bf The Even Condition}]\label{strongeven}
  If a Gauss diagram is realizable then the number of all chords that cross a both of non-intersecting chords and every chord is even (including zero).
\end{corollary}

We conclude this subsection with an explanation why the even condition is not sufficient for realizability of Gauss diagrams.

Roughly speaking, from the proof of Proposition \ref{prop1} it follows that every plane curve can be obtained by attaching its loops to each other by given points. Conversely, if a Gauss diagram satisfies the even condition then it may be non-releasible. Indeed, when we attach a loop, say, $\mathscr{C}(b)$ to a loop $\mathscr{C}(a)$, where $\m{b} \in \m{a}_\parallel$, by given points (=elements of the set $\m{a}_\times \cap \m{b}_\times$) then the loop $\mathscr{C}(b)$ can be self-intersected curve, which means that we get new crossings (= virtual crossings), see {\sc Figure} \ref{rasdutyitrilistnik}.

To be more precisely, we have the following proposition.

\begin{proposition}\label{G=attach}
  Let $\m{G}$ be a non-realizable Gauss diagram which satisfies the even condition. Let $\m{G}$ defines a virtual plane curve $\mathscr{C}$ (up to virtual moves). There exist two non-intersecting chords $\m{a,b \in G}$ such that there are paths $c \to x \to d$, $e \to x \to f$ on a loop $\mathscr{C}(b)$, where $\m{c,d,e,f \in a_\times \cap b_\times}$ are different chords and $x$ is a virtual crossing of $\mathscr{C}$.
\end{proposition}
\begin{proof}
 Let $\m{a,b \in G}$ be two non-intersecting chords. Take non-intersecting $C$-contours $C(\m{a})$, $C(\m{b})$. Hence we may say that the loop $\mathscr{C}(b)$ attaches to the loop $\mathscr{C}(a)$ by the given points $p_1,\ldots,p_n $, where $\{\m{p_1,\ldots,p_n}\}= \m{a_\times \cap b_\times}$. Since $\m{G}$ is not realizable and satisfies the even condition then a virtual crossing may arise only as a self-intersecting point of, say, the loop $\mathscr{C}(b)$. Indeed, when we attach $\mathscr{C}(b)$ to $\mathscr{C}(a)$ by $p_1,\ldots,p_n $ we may get self-interesting points, say, $q_1,\ldots,q_m$ of the loop $\mathscr{C}(b)$. If $\m{G}$ contains all chords $\m{q_1},\ldots,\m{q}_m$ for every such chords $\m{a,b}$, then $\m{G}$ is realizable. Thus, a virtual crossing $x$ does not belong to $\{p_1,\ldots, p_n\} =  \mathscr{C}(a) \cap \mathscr{C}(b)$ for some non-intersecting chords $\m{a,b \in G}$. Then we get two paths $c \to x \to d$, $e \to x \to f$, where $c,d,e,f \in \mathscr{C}(a) \cap \mathscr{C}(b)$ are different chords, as claimed.
\end{proof}

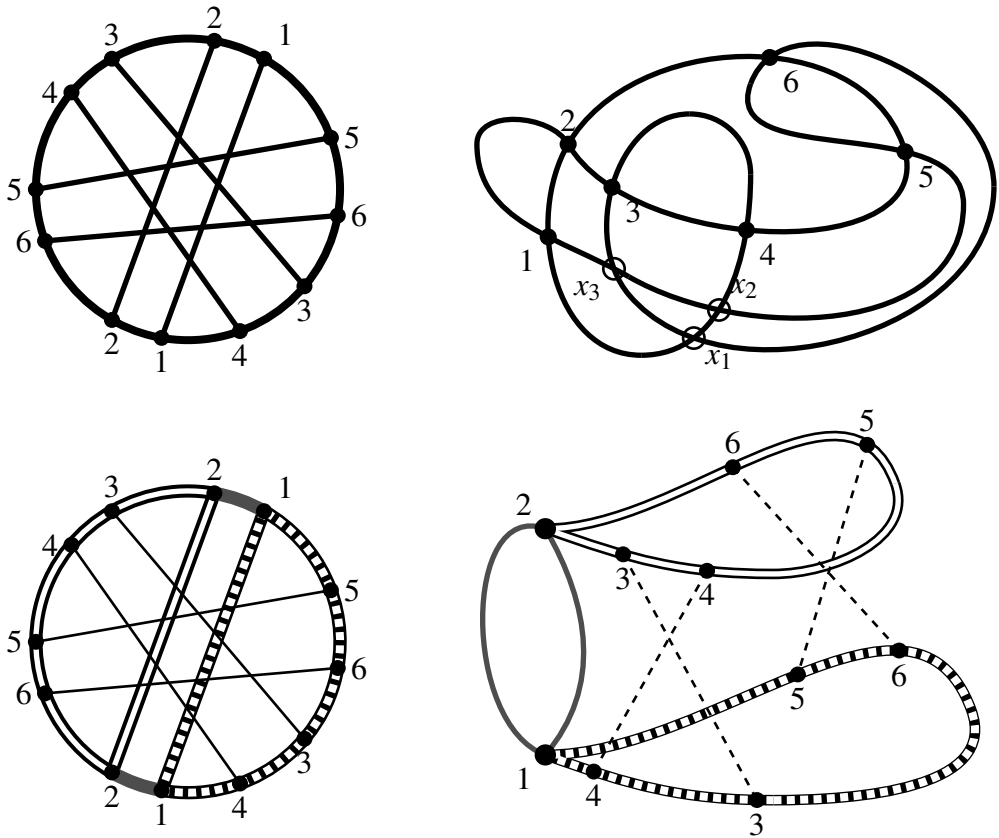
\begin{figure}[h!]
 \begin{center}
 \begin{tikzpicture}[line width =2]
    \draw[line width =3] (0,0) circle (2);
    \draw (60:2) -- (260:2);
    \draw (80:2) -- (240:2);
    \draw (120:2)--(320:2);
    \draw (140:2) -- (290:2);
    \draw (180:2) -- (20:2);
    \draw (200:2) -- (350:2);
    \fill (60:2) circle(3pt) node[above right] {$1$};
    \fill (80:2) circle(3pt) node[above] {$2$};
    \fill (120:2) circle(3pt) node[above] {$3$};
    \fill (140:2) circle(3pt) node[left] {$4$};
    \fill (180:2) circle(3pt) node[left] {$5$};
    \fill (200:2) circle(3pt) node[left] {$6$};
    \fill (240:2) circle(3pt) node[below] {$2$};
    \fill (260:2) circle(3pt) node[below] {$1$};
    \fill (290:2) circle(3pt) node[below] {$4$};
    \fill (320:2) circle(3pt) node[below] {$3$};
    \fill (350:2) circle(3pt) node[right] {$6$};
    \fill (20:2) circle(3pt) node[right] {$5$};
  \begin{scope}[xshift=5cm, yshift = 0.6cm, line width =2,scale = 0.8]
    \draw [name path= a] (0,0) to [out = 300, in = 290] (5.5,0);
    \draw[name path =a2](5.5,0) to [out = 110, in = 60] (0,0);
    \draw[name path = b1] (0,0) to [out = 240, in = 180](1.2,-3.5);
    \draw[name path =b2](1.2, -3.5) to [out = 0, in = 270] (3,-0.5);
    \draw[name path =c] (3,-0.5) to [out = 90, in = 0] (2,0.5);
    \draw[name path =c2](2,0.5) to [out = 180, in =130] (1,-2.5) to [out = 310, in = 270] (7,-0.7);
    \draw[name path =g1] (7,-0.7) to [out = 90, in = 70] (3, 1);
    \draw[name path =g2](3,1) to [out = 250, in = 90] (6.5, -1) to [out = 270, in = 330] (1,-2.2);
    \draw[name path =e] (1,-2.2) to [out = 150, in = 270] (-1.5,0) to [out =90, in = 120] (0,0);
\fill [name intersections={of=e and b1, by={1}}]
(1) circle (4pt) node[below left] {$1$};
\fill [name intersections={of=a and b1, by={2}}]
(2) circle (4pt) node[above] {$2$};
\fill [name intersections={of=a and b2, by={4}}]
(4) circle (4pt) node[below right] {$4$};
\fill [name intersections={of=a and c2, by={3}}]
(3) circle (4pt) node[below right] {$3$};
\fill [name intersections={of=a and g2, by={5}}]
(5) circle (4pt) node[below right] {$5$};
\fill [name intersections={of=a2 and g1, by={6}}]
(6) circle (4pt) node[below right] {$6$};
\draw [line width =1,name intersections={of=c2 and b2, by={x1}}]
(x1) circle (5pt) node[below right] {$x_1$};
\draw [line width =1,name intersections={of=g2 and b2, by={x2}}]
(x2) circle (5pt) node[above right] {$x_2$};
\draw [line width =1,name intersections={of=c2 and e, by={x3}}]
(x3) circle (5pt) node[below left] {$x_3$};
\end{scope}
\begin{scope}[yshift=-6cm]
  \draw[line width =2] (0,0) circle (2);
    \draw[line width =5] (60:2) -- (260:2);
    \draw[white, dotted, line width =3] (60:2) -- (260:2);
    \draw[line width =5] (80:2) -- (240:2);
    \draw[white] (80:2) -- (240:2);
    \draw[line width =1] (120:2)--(320:2);
    \draw[line width =1] (140:2) -- (290:2);
    \draw[line width =1] (180:2) -- (20:2);
    \draw[line width =1] (200:2) -- (350:2);
   \begin{scope}[rotate =60]
   \draw[ black!70, line width=5] (2,0) arc(0:20:2);
   \end{scope}
   \begin{scope}[rotate =80]
   \draw[line width=5] (2,0) arc(0:160:2);
   \draw[ white] (2,0) arc(0:160:2);
   \end{scope}
   \begin{scope}[rotate =240]
   \draw[black!70, line width=5] (2,0) arc(0:20:2);
   \end{scope}
   \begin{scope}[rotate =260]
   \draw[line width=5] (2,0) arc(0:160:2);
   \draw[white, dotted,line width=3] (2,0) arc(0:160:2);
   \end{scope}
    \fill (60:2) circle(3pt) node[above right] {$1$};
    \fill (80:2) circle(3pt) node[above] {$2$};
    \fill (120:2) circle(3pt) node[above] {$3$};
    \fill (140:2) circle(3pt) node[left] {$4$};
    \fill (180:2) circle(3pt) node[left] {$5$};
    \fill (200:2) circle(3pt) node[left] {$6$};
    \fill (240:2) circle(3pt) node[below] {$2$};
    \fill (260:2) circle(3pt) node[below] {$1$};
    \fill (290:2) circle(3pt) node[below] {$4$};
    \fill (320:2) circle(3pt) node[below] {$3$};
    \fill (350:2) circle(3pt) node[right] {$6$};
    \fill (20:2) circle(3pt) node[right] {$5$};
\begin{scope}[xshift = 4.7cm, yshift = 1.5cm]
  \draw[name path = l1,white] (0.3,1)--(3,-4);
  \draw[name path =l2,white] (3,1)--(0.2,-4);
  \draw[name path = l3,white] (2.3,1)--(5,-2);
  \draw[name path =l4,white] (4.5,2)--(3,-3);
  \draw [line width =4, name path= c1] (0,0) to [out = 340, in = 180] (3,-0.6) to [out =0, in = 290] (4.6,0.6);
  \draw [line width =4, name path= c2] (4.6, 0.6) to  [out =110, in = 0] (0,0);
  \draw [white] (0,0) to [out = 340, in = 180] (3,-0.6) to [out =0, in = 290] (4.6,0.6);
  \draw [white] (4.6, 0.6) to  [out =110, in = 0] (0,0);
  \draw [line width =4, name path =o1] (0,-3) to [out = 340, in = 180] (3,-3.6);
  \draw [line width =4, name path =o2] (3,-3.6) to [out =0, in = 290] (5.6,-2.4) to  [out =110, in = 0] (0,-3);
   \draw[white, dotted, line width =3] (0,-3) to [out = 340, in = 180] (3,-3.6);
  \draw [white, dotted, line width =3] (3,-3.6) to [out =0, in = 290] (5.6,-2.4) to  [out =110, in = 0] (0,-3);
  \draw[line width =2,black!70] (0,0) [out = 160, in = 160] to (0,-3);
  \draw[line width =2, black!70] (0,0) [out = 300, in = 50] to (0,-3);
    \fill [name intersections={of=c1 and l1, by={3u}}]
 (3u) circle (3pt) node[below] {$3$};
    \fill [name intersections={of=c1 and l2, by={4u}}]
 (4u) circle (3pt)node [below] {$4$};
    \fill [name intersections={of=o1 and l2, by={4d}}]
 (4d) circle (3pt)node [below] {$4$};
    \fill [name intersections={of=o1 and l1, by={3d}}]
 (3d) circle (3pt)node [below] {$3$};
  \fill [name intersections={of=c2 and l4, by={5u}}]
 (5u) circle (3pt) node[above] {$5$};
 \fill [name intersections={of=o2 and l4, by={5d}}]
 (5d) circle (3pt) node[below] {$5$};
 \fill [name intersections={of=c2 and l3, by={6u}}]
 (6u) circle (3pt) node[above] {$6$};
 \fill [name intersections={of=o2 and l3, by={6d}}]
 (6d) circle (3pt) node[below] {$6$};
 \draw[line width =1,dashed] (3u)--(3d);
 \draw[line width =1,dashed] (4u)--(4d);
 \draw[line width =1,dashed] (5u)--(5d);
 \draw[line width =1,dashed] (6u)--(6d);
 \fill (0,0) circle (4pt) node [above left] {$2$};
 \fill (0,-3) circle (4pt) node[below left] {$1$};
\end{scope}
\end{scope}
\end{tikzpicture}
  \end{center}
\caption{It shows that the even condition is not sufficient for realizability of the Gauss diagrams. We see that the plane curve can be obtained by attaching the white-black loop to the dotted one by the points $3,4,5,6$, and thus the dotted loop has to have ``new'' crossings (= self-intersections) $x_1,x_2,x_3$.}\label{rasdutyitrilistnik}
\end{figure}

\subsection{The Sufficient Condition}

\begin{definition}
  Let $\m{G}$ be a Gauss diagram (not necessarily realizable) and $X(\m{a,b})$ its $X$-contour. Take the $X$-contour coloring of $\m{G}$. A chord of $\m{G}$ is called \textit{colorful for $X(\m{a,b})$} if its endpoints are in arcs which have different colors.
\end{definition}

Similarly, one can define \textit{a colorful chord for a $C$-contour} $C(\m{a})$ of $\m{G}$.

\begin{example}
Let us consider the Gauss diagram, which is shown in {\sc Figure} \ref{d}. One can easy check that this Gauss diagram is not realizable. Let us consider the orange $X$-contour $X(1,3)$ and the $X(1,3)$-coloring of $\m{G}$. The chord with the endpoints $5$ is colorful for the $X$-contour $X(1,3)$. It is interesting to consider the corresponding coloring of the virtual plane curve: one can think that we forget to change color when we cross the gray loop, \textit{i.e.,} the gray loop ``does not divide'' the curve into two parts. We shall show that this observation is typical for every non-realizable Gauss diagram.
\end{example}

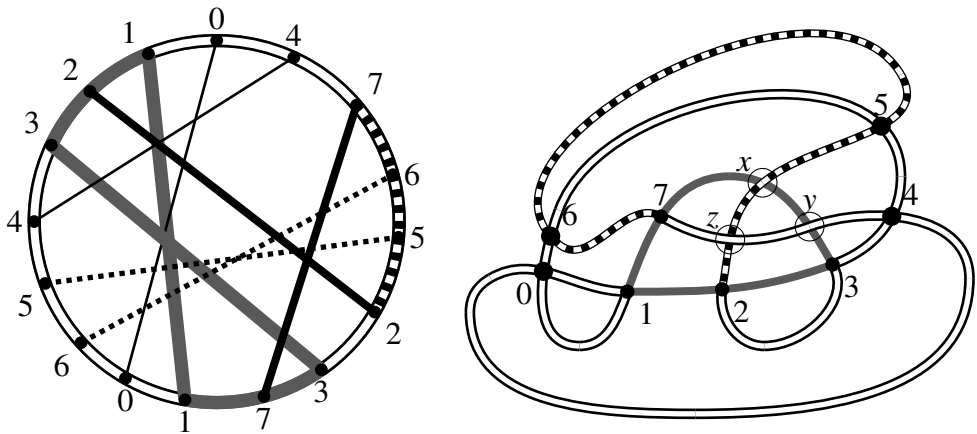
\begin{figure}[h!]
 \begin{tikzpicture}
\begin{scope}[yshift=-0.75cm,xshift=4.5cm,scale=0.9]
  \draw [line width = 4,name path =v1] (-1,0) to [out = 30, in = 180] (1,-0.2);
  \draw [line width = 2, white] (-1,0) to [out = 30, in = 180] (1,-0.2);
  \draw[line width =3, name path= o1,black!65] (1,-0.2) to [out =0, in = 200] (4,0.2);
  \draw [line width = 4, name path= a] (4,0.2) to [out =20 , in = 270] (5,1.5);
  \draw [line width = 2, white] (4,0.2) to [out =20 , in = 270] (5,1.5);
  \draw [line width = 4, name path= b] (5,1.5) to [out = 90, in = 70] (0,1);
  \draw [line width = 2, white] (5,1.5) to [out = 90, in = 70] (0,1);
  \draw [line width = 4, name path =c] (0,1) to [out = 250, in = 180] (0.3,-1);
  \draw [line width = 2, white] (0,1) to [out = 250, in = 180] (0.3,-1);
  \draw [line width = 4] (0.3,-1) to [out = 0, in = 250] (1,-0.2);
  \draw [line width = 2, white] (0.3,-1) to [out = 0, in = 250] (1,-0.2);
  \draw [line width = 3, name path= o2,black!65] (1,-0.2) to [out = 70, in=180] (2.5,1.5);
  \draw [line width = 3, name path= o3,black!65] (2.5,1.5) to [out = 0, in = 120] (4,0.2);
  \draw [line width =4] (4,0.2) to [out = 300, in = 0] (3,-1);
  \draw [line width =2, white] (4,0.2) to [out = 300, in = 0] (3,-1);
  \draw [name path= w, white] (3,-1) to [out = 180, in = 240] (2.5,0.3);
  \fill [name intersections={of=o1 and w, by={T4}}]
 (T4) circle (1pt);
 \draw [name path= h2, white] (T4) to (3.2,2);
  \fill [name intersections = {of =o3 and h2, by ={T2}}]
 (T2) circle (1pt);
 \draw [name path= h3, white] (T4) to (4,1);
  \fill [name intersections = {of =o3 and h3, by ={T3}}]
 (T3) circle (1pt);
 \draw [name path= h4, white] (T4) to (1,1.5);
  \fill [name intersections = {of =o2 and h4, by ={T1}}]
 (T1) circle (1pt);
\draw [line width =4] (3,-1) to [out = 180, in = 260] (T4);
\draw [line width =2, white] (3,-1) to [out = 180, in = 260] (T4);
\draw [line width =3,name path=y1] (T4) to [out = 80, in = 220] (T2);
\draw [line width =2,white, dashed] (T4) to [out = 80, in = 220] (T2);
\draw [line width =3, name path = d] (T2) to [out = 40, in = 300] (5,3);
\draw [line width =2, white, dashed] (T2) to [out = 40, in = 300] (5,3);
\draw [line width =3, name path = d1] (5,3) to [out = 120, in = 140] (0,0.5);
\draw [line width =2, white, dashed] (5,3) to [out = 120, in = 140] (0,0.5);
\draw [line width =3,name path=v2] (0,0.5) to [out = 320, in = 150] (T1);
\draw [line width =2,white, dashed] (0,0.5) to [out = 320, in = 150] (T1);
\draw [line width =4, name path = y2] (T1) to [out = 330, in = 200] (T3);
\draw [line width =2, white] (T1) to [out = 330, in = 200] (T3);
\draw [line width =4, name path = d2] (T3) to [out = 20, in = 90] (6,0);
\draw [line width =2, white] (T3) to [out = 20, in = 90] (6,0);
\draw [line width =4] (6,0) to [out = 270, in = 0] (2,-2);
\draw [line width =2, white] (6,0) to [out = 270, in = 0] (2,-2);
\draw [line width =4] (2,-2) to [out = 180, in = 270] (-1.3,-0.5);
\draw [line width =2, white] (2,-2) to [out = 180, in = 270] (-1.3,-0.5);
\draw [line width =4] (-1.3, -0.5) to [out = 90, in = 210] (-1,0);
\draw [line width =2, white] (-1.3, -0.5) to [out = 90, in = 210] (-1,0);
\fill (T1) circle (3pt) node [above] {$7$};
\draw (T2) circle (6pt) node [above left] {$x$};
\draw (T3) circle (6pt) node [above] {$y$};
\fill (T4) circle (3pt) node [below right] {$2$};
 \draw [name intersections = {of =y1 and y2, by ={I}}]
 (I) circle (6pt) node [above left] {$z$};
  \fill (1,-0.2) circle (3pt) node [below right] {$1$};
  \fill (4,0.2) circle (3pt) node [below right] {$3$};
 \fill [name intersections = {of =v1 and c, by ={a1}}]
 (a1) circle (4pt) node [below left] {$0$};
 \fill [name intersections = {of =a and d2, by ={a2}}]
 (a2) circle (4pt) node [above right] {$4$};
  \fill [name intersections = {of =b and d, by ={a3}}]
 (a3) circle (4pt) node [above] {$5$};
  \fill [name intersections = {of =d1 and c, by ={a4}}]
 (a4) circle (4pt) node [above right] {$6$};
  \end{scope}
  \begin{scope}[scale = 0.8]
    \draw[line width =5, black!65] (0,0) circle (3);
   \begin{scope}[rotate =112]
   \draw[line width =5, black!65] (3,0) arc(0:43:3);
   \end{scope}
  \begin{scope}[rotate =260]
   \draw[line width=5, black!65] (3,0) arc(0:45:3);
   \end{scope}
  \begin{scope}[rotate =155]
   \draw[line width=5, black] (3,0) arc(0:105:3);
   \draw[line width=3, white] (3,0) arc(0:105:3);
  \end{scope}
  \begin{scope}[rotate =305]
   \draw[line width=5] (3,0) arc(0:25:3);
   \draw[line width=3, white] (3,0) arc(0:25:3);
   \end{scope}
  \begin{scope}[rotate =330]
   \draw[line width=5] (3,0) arc(0:70:3);
   \draw[line width=3, white, dashed] (3,0) arc(0:70:3);
   \end{scope}
  \begin{scope}[rotate =40]
   \draw[line width=5] (3,0) arc(0:72:3);
   \draw[line width=3, white] (3,0) arc(0:72:3);
   \end{scope}
   \draw[line width =5, black!65] (112:3) --(260:3);
   \draw[line width =1] (90:3) -- (240:3);
   \draw[line width =5, black!65] (155:3) --(305:3);
   \draw[line width=3] (134:3) -- (330:3);
   \draw[line width=2, dotted] (200:3) --(355:3);
   \draw[line width=2,dotted] (222:3) -- (15:3);
   \draw[line width =3] (285:3) --(40:3);
   \draw[line width=1] (180:3)--(65:3);
    \fill(90:3) circle(3pt) node[above] {$0$};
    \fill(112:3) circle(3pt) node[above left] {$1$};
    \fill(134:3) circle(3pt) node[above left] {$2$};
    \fill(155:3) circle(3pt) node[above left] {$3$};
    \fill(180:3) circle(3pt) node[left] {$4$};
    \fill(200:3) circle(3pt) node[below left] {$5$};
    \fill(222:3) circle(3pt) node[below left] {$6$};
    \fill(240:3) circle(3pt) node[below] {$0$};
    \fill(260:3) circle(3pt) node[below] {$1$};
    \fill(285:3) circle(3pt) node[below] {$7$};
    \fill(305:3) circle(3pt) node[below] {$3$};
    \fill(330:3) circle(3pt) node[below right] {$2$};
    \fill(355:3) circle(3pt) node[right] {$5$};
    \fill(15:3) circle(3pt) node[right] {$6$};
    \fill(40:3) circle(3pt) node[above right] {$7$};
    \fill(65:3) circle(3pt) node[above] {$4$};
  \end{scope}
  \end{tikzpicture}
\caption{This Gauss diagram satisfies even condition but is non-realizable. There are colorful chords (\textit{e.g.} the chord with endpoints $5$).}\label{d}
\end{figure}

We have seen that if a Gauss diagram is realizable then there is no colorful chord, with respect to every $X$-contour. We shall show that it is sufficient condition for realizability of a Gauss diagram.

\begin{proposition}\label{propaboutX}
Let $\m{G}$ be a non-realizable Gauss diagram but satisfy the even condition. Then there exists an $X$-contour and a colorful chord for this $X$-contour.
\end{proposition}
\begin{proof}
By Theorem \ref{virt}, $\m{G}$ defines a virtual curve $\mathscr{C}$ (= the shadow of a virtual knot diagram) up to virtual moves. Starting from a crossing, say, $o$, let us walk along $\mathscr{C}$ till we meet the first virtual crossing, say, $x$. Next, let us keep walking in the same direction till we meet the first real crossing, say $d$. Denote this path by $\mathscr{P}$. Just for convinces, let us put the labels, say, $x_0,x_1$ of the virtual crossing $x$ on the circle of $\m{G}$ in the order we meet them on $\mathscr{P}$.

We can take two real crossings, say, $a,b$ such that the path $a \to x \to b$ does not contain another real crossings and $\m{a,b} \in \m{o}_\times.$ Indeed, from Proposition \ref{G=attach} it follows that for a chord $\m{o}\in \m{G}$ we can find $n \ge 1$ chords $\m{o}^1,\ldots,\m{o}^n$ such that, for every $1\le i \le n$, we have: (1) the chords $\m{o}, \m{o}^i$ do not intersect, (2) the loop $\mathscr{C}(o)$ contains the following paths $a^i \to x \to b^i$, $c^i \to x \to d^i$ where $\m{a}^i,\m{b}^i,\m{c}^i,\m{d}^i \in \m{o}_\times \cap \m{o}^i_\times$ are different chords. Thus, for some $1\le i,j \le n$ we have an arc, say, $\m{a}^i_1\m{b}^j_1$ contains only one of $x_0$ or $x_1$, no endpoints of another chords, and $\m{a}^i,\m{b}^j \in \m{o}_\times$. Denote this arc by $\m{a_1b_1}$ and assume that $x_0$ lies on $\m{a_1b_1}$.

We next construct an $X$-contour contains the arc $\m{a_1b_1}$. To do so, we have to consider the following two cases:

\textit{Case 1.} The chords $\m{a,b}$ intersect.

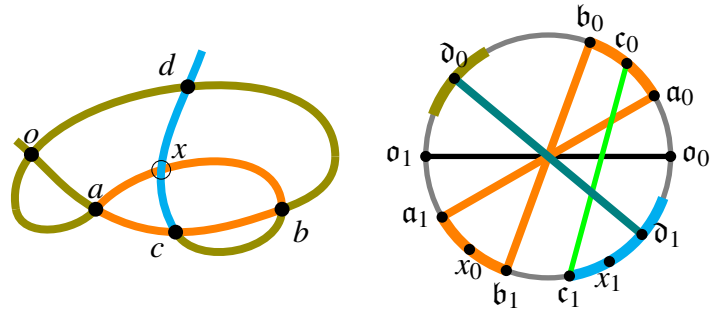
\begin{figure}[h!]
 \begin{tikzpicture}[scale =0.7]
  \draw[line width =3, orange, name path = a] (1.5,0) to [out = 330, in = 200] (5,0);
 \draw[white, name path = a1] (3,-1) -- (3,0);
  \draw [name intersections = {of =a and a1, by ={C}}] (C) circle (2pt) node[above] {};
  \draw[line width =3, olive, name path = b] (1.5,0) to [out = 150, in = 320] (0,1.3);
  \draw[line width =3, olive, name path = c] (5,0) to [out = 20, in = 270] (6,1);
  \draw[line width =3, olive, name path = d] (6,1) to [out = 90, in = 100] (0,0);
  \draw[line width =3, olive, name path = e] (0,0) to [out = 280, in = 230] (1.5,0);
  \draw[line width =3, orange, name path = f] (1.5,0) to [out = 50, in = 90] (5,0);
  \draw[line width =3, olive] (5,0) to [out = 270, in = 300] (C);
  \draw[line width =3, cyan, name path = g] (C) to [out = 120, in = 250] (3.5,3);
  \fill (1.5,0) circle (4pt) node [above] {$a$};
  \fill (5,0) circle (4pt) node [below right] {$b$};
  \fill (C) circle (4pt) node [below left] {$c$};
  \draw [name intersections = {of =g and f, by ={X}}]
 (X) circle (5pt) node [above right] {$x$};
 \fill [name intersections = {of =d and g, by ={D}}]
 (D) circle (4pt) node [above left] {$d$};
 \fill [name intersections = {of =d and b, by ={O}}]
 (O) circle (4pt) node [above] {$o$};
\begin{scope}[xshift = 10cm, yshift = 1cm]
 \draw[line width =2,gray] (0,0) circle (2.3);
\begin{scope}[rotate =30]
   \draw[orange, line width=4] (2.3,0) arc(0:40:2.3);
\end{scope}
\begin{scope}[rotate =120]
   \draw[olive, line width=4] (2.3,0) arc(0:40:2.3);
\end{scope}
\begin{scope}[rotate =280]
   \draw[cyan, line width=4] (2.3,0) arc(0:60:2.3);
\end{scope}
\begin{scope}[rotate =210]
   \draw[orange, line width=4] (2.3,0) arc(0:40:2.3);
\end{scope}
\draw[line width =2] (0:2.3) -- (180:2.3);
  \fill (0:2.3) circle (3pt) node [right] {$\m{o}_0$};
  \fill (180:2.3) circle (3pt) node [left] {$\m{o}_1$};
\draw[line width =3, orange] (30:2.3) -- (210:2.3);
  \fill (30:2.3) circle (3pt) node [right] {$\m{a}_0$};
  \fill (210:2.3) circle (3pt) node [left] {$\m{a}_1$};
\draw[line width =2, green] (50:2.3) -- (280:2.3);
  \fill (50:2.3) circle (3pt) node [above] {$\m{c}_0$};
  \fill (280:2.3) circle (3pt) node [below] {$\m{c}_1$};
\draw[line width =3, orange] (70:2.3) -- (250:2.3);
  \fill (70:2.3) circle (3pt) node [above] {$\m{b}_0$};
  \fill (250:2.3) circle (3pt) node [below] {$\m{b}_1$};
\draw[line width =3, teal] (320:2.3) -- (140:2.3);
  \fill (320:2.3) circle (3pt) node [right] {$\m{d}_1$};
  \fill (140:2.3) circle (3pt) node [above] {$\m{d}_0$};
\fill (230:2.3) circle (3pt) node [below] {$x_0$};
\fill (300:2.3) circle (3pt) node [below] {$x_1$};
\end{scope}
\end{tikzpicture}
\caption{Since the curve $\mathscr{C}$ is determined up to virtual moves then the cyan line has to cross the olive line.}\label{degen1}
\end{figure}

Take the $X$-contour $X(\m{a,b})$ which contains the arc $\m{a_1b_1}$. It is easy to see that this $X$-contour has at least one real door chord, because its another arc $\m{a_0b_0}$ has no virtual crossing thus it has to have at least one real crossing (see {\sc Figure \ref{degen1}}).

\textit{Case 2.} The chords $\m{a,b}$ do not intersect.

\begin{figure}[h!]
 \begin{center}
\begin{tikzpicture}[scale = 0.8]
\begin{scope}
    \draw[line width =3, orange, name path = a] (1.5,0) to [out = 330, in = 200] (5,0);
 \draw[white, name path = a1] (3.5,-1) -- (3.5,0.5);
  \draw [name intersections = {of =a and a1, by ={C}}] (C) circle (2pt) node[above] {};
  \draw[line width =3, olive, name path = b] (1.5,0) to [out = 150, in = 320] (0,1.3);
  \draw[line width =3, olive, name path = c] (5,0) to [out = 20, in = 270] (6,1);
  \draw[line width =3, olive, name path = d] (6,1) to [out = 90, in = 100] (0,0);
  \draw[line width =3, olive, name path = e] (0,0) to [out = 280, in = 230] (1.5,0);
  \draw[line width =3, orange, name path = f] (1.5,0) to [out = 50, in = 90] (5,0);
 \draw[white, name path = a2] (2.3,0) -- (2.3,1);
 \draw [name intersections = {of =f and a2, by ={A}}] (A) circle (2pt) node[above] {};
  \draw[line width =3, olive] (5,0) to [out = 270, in = 300] (C);
  \draw[line width =3, cyan, name path = g] (C) to [out = 120, in = 250] (3.5,3);
  \fill (1.5,0) circle (3pt) node [above] {$\widetilde{o}$};
  \fill (5,0) circle (3pt) node [below right] {$b$};
  \fill (C) circle (3pt) node [below left] {};
  \draw [name intersections = {of =g and f, by ={X}}]
 (X) circle (5pt) node [above right] {$x$};
 \fill [name intersections = {of =d and g, by ={D}}]
 (D) circle (3pt) node [above left] {};
 \fill [name intersections = {of =d and b, by ={O}}]
 (O) circle (3pt) node [above] {$o$};
\fill (A) circle (3pt) node [above] {$a$};
\end{scope}
\begin{scope}[line width =2 ,xshift = 10cm,yshift=1cm]
  \draw[gray, line width =2.5] (0,0) circle (2);
\begin{scope}[rotate =35]
   \draw[orange, line width=4] (2,0) arc(0:25:2);
   \end{scope}
\begin{scope}[rotate =200]
   \draw[orange, line width=4] (2,0) arc(0:60:2);
   \end{scope}
   \draw[cyan] (20:1.5) -- (340:2);
   \fill(340:2) circle(3pt) node[right] {$\m{d}_1$};
   \draw[teal]  (30:1) -- (300:2);
   \fill(320:2) circle(3pt) node[below] {$x_1$};
   \fill(300:2) circle(3pt) node[below] {$\m{c}_1$};
   \draw  (10:2) -- (180:2);
   \fill(10:2) circle(3pt) node[right] {$\m{o}_0$};
   \fill(180:2) circle(3pt) node[left] {$\m{o}_1$};
\draw [line width = 3, orange] (35:2) -- (200:2);
   \fill(35:2) circle(3pt) node[right] {$\widetilde{\m{o}}_0$};
   \fill(200:2) circle(3pt) node[left] {$\widetilde{\m{o}}_1$};
\draw[line width =3, orange]  (60:2) -- (260:2);
   \fill(60:2) circle(3pt) node[above right] {$\m{b}_0$};
   \fill(260:2) circle(3pt) node[below] {$\m{b}_1$};
\draw[green]  (170:1.3) -- (220:2);
   \fill(220:2) circle(3pt) node[below] {$\m{a}_1$};
  \fill (240:2) circle(3pt) node[below] {$x_0$};
\end{scope}
 \end{tikzpicture}
\end{center}
\caption{The ``general'' position of chords $\m{a,b,c,d}$ is shown. Our walk along the path $\mathscr{P}$ correspondences to $\m{o_0 \to \widetilde{o}_0 \to b_0\to \cdots \to \m{d_1}}$.}\label{last}
\end{figure}
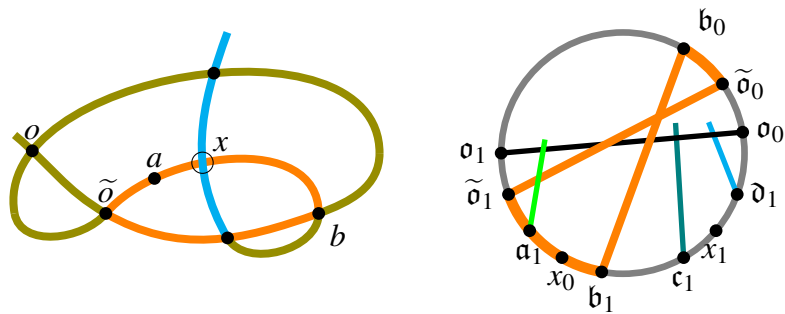

Let us walk along the circle of $\m{G}$ in the direction $\m{b_1} \to x_0 \to \m{a}_1$ ({\sc Figure} \ref{last}) till we meet the first chord, say, $\widetilde{\m{o}}$ such that $\m{a,b,o} \in \widetilde{\m{o}}_\times$. If we cannot find such chord we set $\widetilde{\m{o}}:=\m{o}$. Take the $X$-contour $X(\widetilde{\m{o}},\m{b})$ contains the arc $\m{a_1b_1}$. This $X$-contour is non-degenerate. Indeed, let $\widetilde{\m{o}}:=\m{o}$ and endpoints of all chords, which start in the arc $\m{b_1}x_1\m{d}_1$, lie on the arc $\m{o}_0\m{b_0}$. We then get the situation is shown in {\sc Figure} \ref{last2}.
\begin{figure}[h!]
  \begin{tikzpicture}[scale = 0.8]
    \draw[line width =3, orange, name path = a] (0,0) to [out = 330, in = 200] (5,0);
    \draw[line width =3, orange, name path = f] (0,0) to [out = 50, in = 90] (5,0);
    \draw[line width =3, olive] (5,0) to [out = 20, in = 200] (5.5,0.2);
    \draw[line width =3, olive] (0,0) to [out = 150, in = 320] (-0.3,0.2);
    \draw[line width =3, olive] (-0.2,-0.3) to [out = 60, in = 230] (0,0);
 \draw[white, name path = a1] (3.5,-1) -- (3.5,0.5);
 \draw [name intersections = {of =a and a1, by ={C}}] (C) circle (2pt) node[above] {};
  \draw[white, name path = a2] (1.5,-1) -- (1.5,0.5);
  \draw [name intersections = {of =a and a2, by ={D}}] (D) circle (2pt) node[above] {};
   \draw[line width =3, olive] (5,0) to [out = 270, in = 300] (C);
   \draw[line width =3, cyan, name path = g] (C) to [out = 120, in = 180] (5,2);
   \draw[line width =3,cyan] (5,2) to [out = 0, in = 90] (6,0.5);
   \draw[line width =3, cyan] (6,0.5) to [out = 270, in = 270] (D);
   \draw[line width =3, cyan] (D) to [out = 90, in = 240] (1.7,0.1);
  \draw[line width =2, blue, dotted] (C) to [out = 90, in = 90] (D);
  \draw [name intersections = {of =g and f, by ={X}}]
   (X) circle (5pt) node [above left] {$x$};
 \fill (C) circle (3pt) node [below left] {$c$};
 \fill (D) circle (3pt) node [below left] {$d$};
 \fill (0,0) circle (3pt) node [above] {$o$};
 \fill (5,0) circle (3pt) node [below right] {$b$};
\end{tikzpicture}
\caption{The dotted blue line gives another path from $c$ to $d$.}\label{last2}
\end{figure}
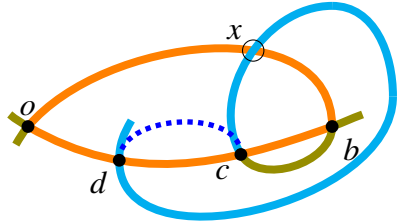

It follows that we can take another path $\mathscr{P}$ without the virtual crossing $x$. This contradiction implies that the $X$-contour $X(\m{o,b})$ is non-degenerate.

So, we have a non-degenerate $\mathscr{X}$-contour of the curve $\mathscr{C}$ such that $x$ is its virtual door. If the $X$-contour of $\m{G}$ contains another chords such that all their endpoints lie on an arc of this $X$-contour, then Proposition \ref{RemConway} implies that Conway's smoothing all such chords does not change the $\mathscr{X}$-contour coloring of $\mathscr{P}$. Thus we may assume that this $\mathscr{X}$-contour is the closed curve without self-intersections. Hence, by Jordan curve Theorem, it divides $\mathscr{P}$ into two parts. Let us color these parts into two different colors. If the crossing $x$ would be a real door of this $\mathscr{X}$-contour, then by Definition \ref{coloring}, we can take such $\mathscr{X}$-contour coloring of $\mathscr{P}$ as before. But since $x$ is not real door then after meeting it at the second time we do not change the color and thus we get an intersection point of two lines which have different colors, \textit{i.e.,} the corresponding chord is colorful. This completes the proof.
\end{proof}

\begin{lemma}\label{prop1}
  Let $\m{G}$ be a Gauss diagram. Consider a $C(\m{a})$-contour coloring of $\m{G}$ for some chord $\m{a\in G}$. If there exists a colorful chord for the $C$-contour $C(\m{a})$ then the diagram $\m{G}$ does not satisfy the even condition.
\end{lemma}
\begin{proof}
Indeed, let $\m{b}$ be a colorful chord for the $C$-contour $C(\m{a})$. First note that, the chord $\m{b}$ cannot cross $\m{a}$ because otherwise $\m{b}$ should be a door chord of the $C$-contour $C(\m{a})$. Next, if the chord $\m{b}$ is colorful then it crosses an odd number of door chords of $C(\m{a})$. Hence $|\m{a}_\times \cap \m{b}_\times| \equiv 1 \bmod 2$, as claimed.
\end{proof}

\begin{figure}[h!]
 \begin{tikzpicture}[scale=0.8]
  \draw [line width = 2,name path =v1, olive] (-1,0) to [out = 30, in = 180] (1,-0.2);
  \draw[line width =2, name path= o1,orange] (1,-0.2) to [out =0, in = 200] (4,0.2);
  \draw [line width = 2, name path= a, olive] (4,0.2) to [out =20 , in = 270] (5,1.5);
  \draw [line width = 2, name path= b, olive] (5,1.5) to [out = 90, in = 70] (0,1);
  \draw [line width = 2, name path =c,olive] (0,1) to [out = 250, in = 180] (0.3,-1);
  \draw [line width = 2, olive] (0.3,-1) to [out = 0, in = 250] (1,-0.2);
  \draw [line width = 2, name path= o2, orange] (1,-0.2) to [out = 70, in=180] (2.5,1.5);
  \draw [line width = 2, name path= o3, orange] (2.5,1.5) to [out = 0, in = 120] (4,0.2);
  \draw [line width =2, name path = ol2, olive] (4,0.2) to [out = 300, in = 0] (3,-1);
  \draw [name path= w, white] (3,-1) to [out = 180, in = 240] (2.5,0.3);
  \fill [name intersections={of=o1 and w, by={T4}}]
 (T4) circle (1pt);
 \draw [name path= h2, white] (T4) to (3.2,2);
  \fill [name intersections = {of =o3 and h2, by ={T2}}]
 (T2) circle (1pt);
 \draw [name path= h3, white] (T4) to (4,1);
  \fill [name intersections = {of =o3 and h3, by ={T3}}]
 (T3) circle (1pt);
 \draw [name path= h4, white] (T4) to (1,1.5);
  \fill [name intersections = {of =o2 and h4, by ={T1}}]
 (T1) circle (1pt);
\draw [line width =2, olive] (3,-1) to [out = 180, in = 260] (T4);
\draw [line width =2,name path=y1, cyan] (T4) to [out = 80, in = 220] (T2);
\draw [line width =2, name path = d, cyan] (T2) to [out = 40, in = 300] (5,3);
\draw [line width =2, name path = d1, cyan] (5,3) to [out = 120, in = 140] (0,0.5);
\draw [line width =2,name path=v2, cyan] (0,0.5) to [out = 320, in = 150] (T1);
\draw [line width =2, name path = y2, olive] (T1) to [out = 330, in = 200] (T3);
\draw [line width =2, name path = d2, olive] (T3) to [out = 20, in = 90] (6,0);
\draw [line width =2, olive] (6,0) to [out = 270, in = 0] (2,-2);
\draw [line width =2, olive] (2,-2) to [out = 180, in = 270] (-1.3,-0.5);
\draw [line width =2, olive] (-1.3, -0.5) to [out = 90, in = 210] (-1,0);
\fill (T1) circle (3pt) node [above] {$7$};
\draw (T2) circle (5pt) node [above] {$x$};
\draw (T3) circle (5pt) node [above] {$y$};
\fill (T4) circle (3pt) node [below right] {$2$};
 \draw [name intersections = {of =y1 and y2, by ={I}}]
 (I) circle (5pt) node [above left] {$z$};
  \fill (1,-0.2) circle (3pt) node [below right] {$1$};
    \fill[white] (4,0.2) circle (7pt);
    \draw[name path = circ, white] (4,0.2) circle (7pt);
\draw [name intersections = {of= circ and o1, by ={3ul}}]
 (3ul) node {};
\draw [name intersections = {of= circ and o3, by ={3ur}}]
 (3ur) node {};
\draw[orange,line width =2] (3ul) to [out = 17, in = 300] (3ur);
\draw [name intersections = {of= circ and a, by ={3dr}}]
 (3dr) node {};
\draw [name intersections = {of= circ and ol2, by ={3dl}}]
 (3dl) node {};
\draw[olive,line width =2] (3dl) to [out = 105, in =200] (3dr);
 \fill [name intersections = {of =v1 and c, by ={a1}}]
 (a1) circle (4pt) node [below left] {$0$};
 \fill [name intersections = {of =a and d2, by ={a2}}]
 (a2) circle (4pt) node [above right] {$4$};
  \fill [name intersections = {of =b and d, by ={a3}}]
 (a3) circle (4pt) node [above] {$5$};
  \fill [name intersections = {of =d1 and c, by ={a4}}]
 (a4) circle (4pt) node [above right] {$6$};
 \begin{scope}[xshift=10cm, yshift=1cm]
   \draw[line width =2] (0,0) circle (3);
   \draw[line width =5,orange] (112:3) --(200:3);
   \draw[line width =1.5,olive] (90:3) -- (222:3);
   \draw[line width=3,green] (134:3) -- (330:3);
   \draw[line width=2,teal] (260:3) --(355:3);
   \draw[line width=2,teal] (240:3) -- (15:3);
   \draw[line width =1.5,olive] (285:3) --(65:3);
   \draw[line width=3,green] (180:3)--(40:3);
  \begin{scope}[rotate =112]
   \draw[orange, line width=5] (3,0) arc(0:88:3);
   \end{scope}
  \begin{scope}[rotate =200]
   \draw[olive, line width=5] (3,0) arc(0:130:3);
   \end{scope}
  \begin{scope}[rotate =330]
   \draw[cyan, line width=5] (3,0) arc(0:70:3);
   \end{scope}
  \begin{scope}[rotate =40]
   \draw[olive, line width=5] (3,0) arc(0:72:3);
   \end{scope}
    \fill(90:3) circle(3pt) node[above] {$0$};
    \fill(112:3) circle(3pt) node[above left] {$1$};
    \fill(134:3) circle(3pt) node[above left] {$2$};
    \fill(180:3) circle(3pt) node[left] {$7$};
    \fill(200:3) circle(3pt) node[below left] {$1$};
    \fill(222:3) circle(3pt) node[below left] {$0$};
    \fill(240:3) circle(3pt) node[below] {$6$};
    \fill(260:3) circle(3pt) node[below] {$5$};
    \fill(285:3) circle(3pt) node[below] {$4$};
    \fill(330:3) circle(3pt) node[below right] {$2$};
    \fill(355:3) circle(3pt) node[right] {$5$};
    \fill(15:3) circle(3pt) node[right] {$6$};
    \fill(40:3) circle(3pt) node[above right] {$7$};
    \fill(65:3) circle(3pt) node[above] {$4$};
 \end{scope}
\end{tikzpicture}
  \caption{The Gauss diagram does not satisfy the even condition; both the chords $1$, $6$ are crossed by only one chord $2$.}\label{del2}
\end{figure}
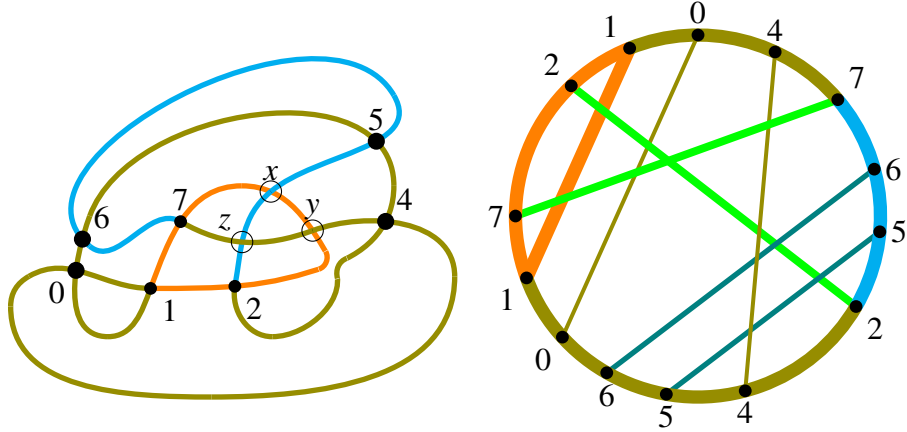

\begin{example}
  Let us consider the Gauss diagram which is shown in {\sc Figure} \ref{del2}. We have the orange $C$-contour $C(1)$ and the $C(1)$-contour coloring of the Gauss diagram and the corresponding (virtual) curve. We see that there are two chords (namely $5$ and $6$) which are colorful and $5_\times \cap 1_\times = 6_\times \cap 1_\times = \{2\}$.
\end{example}

\begin{lemma}\label{prop2}
  Let $\m{G}$ be a Gauss diagram and $\m{a}, \m{b} \in \m{G}$ be its intersecting chords. Suppose that there exists a colorful chord $\m{c}$ for an $X$-contour $X(\m{a},\m{b})$. Then there exists a $C$-contour of the Gauss diagram $\widehat{\m{G}}_\m{b}$ (= Conway's smoothing the chord $\m{b}$) such that the chord $\m{c}$ is colorful for this $C$-contour in $\widehat{\m{G}}_\m{b}$.
\end{lemma}
\begin{proof}
Indeed, consider the Gauss diagram $\widehat{\m{G}}_\m{b}$. From Proposition \ref{RemConway} it follows that after Conway's smoothing the chord $\m{b}$, the chord $\m{c}$ does not intersect $\m{a}$ and intersects the same door chords of the $X$-contour $X(\m{a},\m{b})$ as in $\m{G}$. Further, let us consider the $C$-contour $C(\m{a})$ in $\widehat{\m{G}}_\m{b}$ such that it does not contain the chord $\m{c}$. By Proposition \ref{RemConway}, the chord $\m{a}$ crosses in $\widehat{\m{G}}_\m{b}$ only the chord that are door chords of the $X$-contour $X(\m{a},\m{b})$. Hence, by Definition \ref{coloring}, we may take the $C(\m{a})$-contour coloring of $\widehat{\m{G}}_\m{b}$ such that $\m{c}$ is the colorful chord for this $C$-contour.
 \end{proof}

\begin{proposition}\label{gen}
   Let a Gauss diagram $\mathfrak{G}$ satisfy the even condition. Then $\m{G}$ is realizable if and only if $\widehat{\m{G}}_\m{c}$ satisfies the even condition for every chord $\m{c} \in \m{G}$, $\widehat{\m{G}}_\m{c}$.
\end{proposition}
\begin{proof}
Indeed, let $\mathfrak{G}$ be a non-realizable Gauss diagram and let $\m{G}$ satisfy the even condition. By Proposition \ref{propaboutX}, there exists a colorful chord (say $\m{c}$) for a $X$-counter $X(\m{a},\m{b})$ of $\m{G}$. By Lemma \ref{prop2}, the chord $\m{c}$ is the colorful chord in $\widehat{\m{G}}_\m{b}$. Hence from Lemma \ref{prop1} it follows that $\widehat{\m{G}}_\m{b}$ does not satisfy the even condition, and the statement follows.
\end{proof}

We can summarize our results in the following theorem.

\begin{theorem}\label{themain}
 A Gauss diagram $\m{G}$ is realizable if and only if the following conditions hold:
  \begin{itemize}
    \item[(1)] the number of all chords that cross a both of non-intersecting chords and every chord is even (including zero),
    \item[(2)] for every chord $\m{c} \in \m{G}$ the Gauss diagram $\widehat{\m{G}}_\m{c}$ (= Conway's smoothing the chord $\m{c}$) also satisfies the above condition.
  \end{itemize}
\end{theorem}

\section{Matrixes of Gauss Diagrams}

In this section we are going to translate into matrix language the conditions of realizability of Gauss diagrams. We consider all matrixes over the field $\mathbb{Z}_2$.

\begin{definition}
  Given a Gauss diagram $\m{G}$ contains $n$ chords, say, $\m{c}_1,\ldots,\m{c}_n$. Construct a $n \times n$ matrix $\mathsf{M}(\m{G}) = (\mathsf{m}_{i,j})_{1 \le i,j\le n}$ as follows: (1) put $\mathsf{m}_{ii} = 0$, $1 \le i \le n$, (2) $\mathsf{m}_{ij} = 1$ if the chords $\m{c}_i$, $\m{c}_j$ intersect and $\mathsf{m}_{ij} = 0$ in otherwise.
\end{definition}

It is obviously that $\mathsf{M}(\m{G})$ is a symmetric matrix with respect to the main diagonal.

Next, let $\mathsf{M}(\m{G})$ be a $n\times n$ matrix of a Gauss diagram $\m{G}$. Let $\cup_{1 \le i \le n}\{\mathsf{m}_i = (\mathsf{m}_{i1}, \ldots, \mathsf{m}_{in})\}$ be all its strings. Define a scalar product of strings as follows:
\[
 \langle \mathsf{m}_i, \mathsf{m}_j \rangle:=\mathsf{m}_{i1}\mathsf{m}_{j1} + \cdots + \mathsf{m}_{in}\mathsf{m}_{jn}.
\]

It is clear that we can define ``scalar product'' of any number of strings as follows
\[
 \langle \mathsf{m}_{i_1}, \ldots, \mathsf{m}_{i_k} \rangle: = \mathsf{m}_{i_11} \cdots \mathsf{m}_{i_k1} + \cdots + \mathsf{m}_{i_1n} \cdots \mathsf{m}_{i_kn}.
\]

\begin{lemma}\label{nscalar}
  Let $\m{G}$ be a Gauss diagram contains $n$ chords $\m{c}_1,\ldots,\m{c}_n$. Consider the corresponding matrix $\mathsf{M}(\m{G})$. Let the chords $\m{c}_1,\ldots,\m{c}_n$ correspondence to the strings $\mathsf{m}_1,\ldots,\mathsf{m}_n$, respectively.
     \[
   | {\m{c}_{i_1}}_\times \cap \cdots \cap {\m{c}_{i_k}}_\times| = \langle \mathsf{m}_{i_1}, \ldots, \mathsf{m}_{i_k} \rangle,
  \]
\end{lemma}
for every $\{i_1,\ldots,i_k\} \subseteq \{1,2,\ldots,n\}$.
\begin{proof}
  Indeed, let $\mathsf{m}_{i_1t} \cdots \mathsf{m}_{i_kt} \ne 0$, for some $1 \le t \le n$, then the chord $\m{c}_t$ intersect chords $\m{c}_{i_1}, \ldots, \m{c}_{i_k}$. This completes the proof.
\end{proof}

\begin{theorem}\label{Gauss->matrix}
  Let $\m{G}$ be a realizable Gauss diagram contains $n$ chords. Then its matrix $\mathsf{M}(\m{G})$ satisfies the following conditions:
  \begin{itemize}
    \item[(1)] $\langle \mathsf{m}_i, \mathsf{m}_i \rangle \equiv 0 \bmod(2)$, $1 \le i \le n$,
    \item[(2)]  $\langle \mathsf{m}_i, \mathsf{m}_j \rangle \equiv 0 \bmod(2)$, if the corresponding chords do not intersect,
    \item[(3)] $\langle \mathsf{m}_i, \mathsf{m}_j \rangle + \langle \mathsf{m}_i, \mathsf{m}_k \rangle + \langle \mathsf{m}_j, \mathsf{m}_k \rangle  \equiv 1 \bmod(2)$, if the corresponding chords pairwise intersect.
  \end{itemize}
\end{theorem}
\begin{proof}
 The first two conditions follow from the first condition of Theorem \ref{themain} and Lemma \ref{nscalar}.

Assume that $\m{G}$ does not contain three pairwise intersecting chords. Then $\m{G}$ is realizable if and only if the first condition of Theorem \ref{themain} holds and we get the first two conditions for $\mathsf{M}(\m{G})$.

Next, let $\m{a},\m{b},\m{c}$ be three pairwise intersecting chords. Consider the following sets of chords (see {\sc Figure} \ref{ABC}):
  \begin{align*}
  &A: = \m{a}_\times \setminus \{ \m{b}_\times, \m{c}_\times\},\\
  & B:= \m{b}_\times \cap \m{c}_\times \setminus \left\{ \{\m{a}\}, \m{a}_\times \right\} ). \end{align*}

 \begin{figure}[h!]
   \begin{tikzpicture}
     \draw[line width =2] (0,0) circle (2);
      \draw[line width =1] (90:2)--(270:2);
      \fill(270:2) circle(3pt) node[below] {$\m{b}$};
      \fill(90:2) circle(3pt) node[above] {$\m{b}$};
      \draw[line width =1] (120:2)--(300:2);
      \fill(120:2) circle(3pt) node[above] {$\m{c}$};
      \fill(300:2) circle(3pt) node[below] {$\m{c}$};
      \draw[line width =1] (180:2)--(20:2);
      \fill(180:2) circle(3pt) node[left] {$\m{a}$};
      \fill(20:2) circle(3pt) node[right] {$\m{a}$};
      \draw[line width =8] (40:2) [out = 230, in = 150] to (340:2);
      \fill(40:2) circle(0.2pt) node[right] {${A}$};
      \fill(340:2) circle(0.2pt) node[right] {${A}$};
      \draw[line width =6] (70:2) [out = 250, in = 340] to (150:2);
      \fill(70:2) circle(0.2pt) node[above] {${B}$};
      \fill(150:2) circle(0.2pt) node[left] {${B}$};
      \end{tikzpicture}
    \caption{The sets $A$, $B$ are roughly shown.}\label{ABC}
  \end{figure}
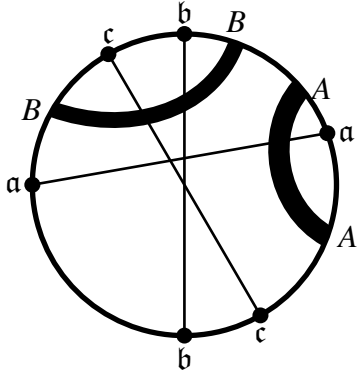

By Proposition \ref{RemConway}, chords $\m{b}, \m{c}$ does not intersect in  $\widehat{\m{G}}_\m{a}$, and the set $\m{b}_\times \cap \m{c}$ of chords in  $\widehat{\m{G}}_\m{a}$ is equal to the set $A \cup B$. Hence, by Theorem \ref{themain},
\[
  |A| + |B| \equiv 0 \bmod (2).
 \]

Further, using Lemma \ref{nscalar}, we obtain:
 \begin{align*}
   &|A| = \langle \m{a}, \m{a} \rangle - \langle \m{a}, \m{b} \rangle - \langle \m{a}, \m{c} \rangle - \langle \m{a}, \m{b}, \m{c} \rangle,\\
   &|B| = \langle \m{b}, \m{c} \rangle -1- \langle \m{a}, \m{b}, \m{c} \rangle.
 \end{align*}

It follows that
 \begin{eqnarray*}
   |A| + |B| & =& \langle \m{a}, \m{a} \rangle - \langle \m{a}, \m{b} \rangle - \langle \m{a}, \m{c} \rangle - \langle \m{a}, \m{b}, \m{c} \rangle\\
    &&+ \langle \m{b}, \m{c} \rangle -1- \langle \m{a}, \m{b}, \m{c} \rangle\\
    &\equiv&  \langle \m{a}, \m{b} \rangle + \langle \m{a}, \m{c} \rangle + \langle \m{b}, \m{c} \rangle \\
    &\equiv& 1 \bmod (2),
 \end{eqnarray*}
because, as we have already discussed, $\langle \m{a}, \m{a} \rangle \equiv 0 \bmod (2)$. As claimed.
\end{proof}

\begin{remark}
  The number of matrix satisfy the above conditions is bigger than a number of realizable Gauss diagrams. The reason is the following observation. A matrix which satisfies the above conditions ``knows'' only intersections but don't know positions of chords. However there is a very important sort of plane curves (=meanderes) such that there is a one-to-one correspondence between Gauss diagrams of these curves and the corresponding matrixes.
\end{remark}

\section{Thurston Generators of Braid Groups}
We will use as generators for $B_n$ the set of {\it positive crossings}, that is, the crossings between two (necessary adjacent) strands, with the front strand having a positive slope. We denote these generators by $\sigma_1, \ldots, \sigma_{n-1}$. These generators are subject to the following relations:
\[
\begin{cases}
\sigma_i\sigma_j = \sigma_j\sigma_i, \mbox{ if } |i-j| >1,\\
\sigma_i\sigma_{i+1}\sigma_i = \sigma_{i+1}\sigma_i\sigma_{i+1}.
\end{cases}
\]

One obvious invariant of an isotopy of a braid is the permutation it induces on the order of the strands: given a braid $B$, the strands define a map $p(B)$ from the top set of endpoints to the bottom set of endpoints, which we interpret as a permutation of $\{1, \ldots, n\}$. In this way we get a homomorphism $p:B_n \to \m{S}_n$, where $\m{S}_n$ is the symmetric group. The generator $\sigma_i$ is mapped to the transposition $s_i = (i,i+1)$. We denote by $S_n = \{s_1, \ldots, s_{n-1}\}$ the set of generators for the symmetric group $\m{S}_n$.

\par Now we want to define an inverse map $p^{-1}:\m{S}_n \to B_n$.  To this end, we need the following definition \cite[p.183]{EpThur}

\begin{definition}\label{conR}
Let $S = \{s_1, \ldots, s_{n-1}\}$ be the set of generators for $\m{S}_n$. Each permutation $\pi$ gives rise to a total order relation $\le_\pi$ on $\{1, \ldots,n\}$  with $i \le_\pi j$ if $\pi(i) < \pi(j)$. We set
$$
R_\pi: = \{(i,j) \in \{1, \ldots, n\} \times \{1,\ldots,n\}|i<j, \, \pi(i) > \pi(j)\}.
$$
\end{definition}

\begin{lemma}\cite[Lemma 9.1.6]{EpThur}\label{criteria}
A set $R$ of pairs $(i,j)$, with $i <j$, comes from some permutation if and only if the following two conditions are satisfied:
\par \textup{i)} If $(i,j) \in R$ and $(j,k) \in R$, then $(i,k) \in R$.
\par \textup{ii)} If $(i,k) \in R$, then $(i,j) \in R$ or $(j,k) \in R$ for every $j$ with $i<j<k$.
\end{lemma}

\par Now we will define (\cite[p. 186]{EpThur}) a very important concept of non-repeating braid (=simple braid).

\begin{definition}
Recall that our set of generators $S_n$ includes only positive crossings; the positive braid monoid is denoted by $B^+  =B_n^+$. We call a positive braid non-repeating if any two of its strands cross at most once. We define $D =D_n \subset B^+_n$ as the set of classes of non-repeating braids.
\end{definition}

The following lemma summarizes all the above mentioned concepts and notations.

\begin{lemma}\cite[Lemma 9.1.10 and Lemma 9.1.11]{EpThur}\label{9.1.11}
The homomorphism $p:B_n^+ \to \mathfrak{S}_n$ is restricted to a bijection $D \to \mathfrak{S}_n$. A positive braid $B$ is non-repeating if and only if $|B| = |p(B)|$ (here $|?|$ means the length of a word). If a non-repeating braid maps to a permutation $\pi$, two strands $i$ and $j$ cross if and only if $(i,j) \in R_\pi$.
\end{lemma}

\begin{example}
Let us consider the following permutation
\[
\pi = \begin{pmatrix}1&2&3&4&5&6 \\ 4&2&6&1&5&3 \end{pmatrix} \in \m{S}_6.
\]

We obtain

\begin{align*}
& \begin{cases}1<2,\\ \pi(1) > \pi(2),\end{cases} \quad \begin{cases}1<4, \\ \pi(1) > \pi(4),\end{cases} \quad  \begin{cases}1<6, \\ \pi(1) > \pi(6),\end{cases} \quad  \begin{cases}2<4, \\ \pi(2) > \pi(4),\end{cases}\\
& \begin{cases}3<4,\\ \pi(3) > \pi(4),\end{cases} \quad \begin{cases}3<5,\\ \pi(3) > \pi(5),\end{cases} \quad \begin{cases}3<6,\\ \pi(3) > \pi(6), \end{cases} \quad \begin{cases}5<6,\\ \pi(5) > \pi(6).\end{cases}
\end{align*}
hence

\[
R(\pi) =  \{(1,2),(1,4),(1,6),(2,4),(3,4),(3,5),(3,6),(5,6)\},
\]
and we get a non-repeating braid is shown in {\sc Figure} \ref{dem1}.

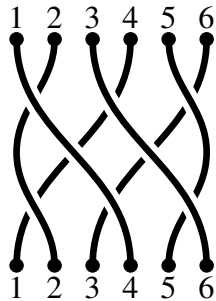
\begin{figure}[h!]
\begin{center}
\begin{tikzpicture}
\draw [line width = 2.5] (2.5,0) to [out = 270, in = 90] (1,-3);
\draw [line width = 7,white] (2,0) to [out = 270, in = 90] (2.5,-1.5) to [out = 270, in = 90] (2,-3);
\draw [line width = 2.5] (2,0) to [out = 270, in = 90] (2.5,-1.5) to [out = 270, in = 90] (2,-3);
\draw [line width = 7,white] (1.5,0) to [out = 270, in = 90] (0,-3);
\draw [line width = 2.5] (1.5,0) to [out = 270, in = 90] (0,-3);
\draw [line width = 7,white] (1,0) to [out = 270, in = 90] (2.5,-3);
\draw [line width = 2.5] (1,0) to [out = 270, in = 90] (2.5,-3);
\draw [line width = 7,white] (0.5,0) to [out = 270, in = 90] (0,-1.5) to [out = 270, in = 90] (0.5,-3);
\draw [line width = 2.5] (0.5,0) to [out = 270, in = 90] (0,-1.5) to [out = 270, in = 90] (0.5,-3);
\draw [line width = 7,white] (0,0) to [out = 270, in = 90] (1.5,-3);
\draw [line width = 2.5] (0,0) to [out = 270, in = 90] (1.5,-3);
\draw [fill](0,0) circle (2.5pt);
\draw [fill](0.5,0) circle (2.5pt);
\draw [fill](1,0) circle (2.5pt);
\draw [fill](1.5,0) circle (2.5pt);
\draw [fill](2,0) circle (2.5pt);
\draw [fill](2.5,0) circle (2.5pt);
\draw [fill](0,-3) circle (2.5pt);
\draw [fill](0.5,-3) circle (2.5pt);
\draw [fill](1,-3) circle (2.5pt);
\draw [fill](1.5,-3) circle (2.5pt);
\draw [fill](2,-3) circle (2.5pt);
\draw [fill](2.5,-3) circle (2.5pt);
\node[above] at (0,0){$1$};
\node[above] at (0.5,0){$2$};
\node[above] at (1,0){$3$};
\node[above] at (1.5,0){$4$};
\node[above] at (2,0){$5$};
\node[above] at (2.5,0){$6$};
\node[below] at (0,-3){$1$};
\node[below] at (0.5,-3){$2$};
\node[below] at (1,-3){$3$};
\node[below] at (1.5,-3){$4$};
\node[below] at (2,-3){$5$};
\node[below] at (2.5,-3){$6$};
\end{tikzpicture}
\caption{The simple braid $R_\pi$ is shown.}\label{dem1}
\end{center}
\end{figure}
\end{example}

\section{Meanders and its Gauss Diagrams}
In this section we deal with closed meanders. We show that any closed meander defines a special sort of Gauss diagrams (= Gauss diagrams of meanders) and then we will see that these diagrams are completely determined by its matrixes (= meander matrix) {\it i.e.,} there is a one-to-one correspondence between meander matrixes satisfy the conditions of Theorem \ref{Gauss->matrix}. It allows us to describe an algorithm constructs all closed meanders.

\begin{definition}
  A (closed) meander is a planar configuration consisting of a simple closed curve and on orientied line, that cross finitely many times and intersect only transversally. Two meanders are equivalent if there exists a homeomorphism of the plane that maps one to the other.
\end{definition}

\begin{figure}
  \begin{tikzpicture}
    \draw [line width = 2] (-0.3,0) to  (7.3,0);
    \draw [fill](0,0) circle (2.5pt);
    \draw [fill](1,0) circle (2.5pt);
    \draw [fill](2,0) circle (2.5pt);
    \draw [fill](3,0) circle (2.5pt);
    \draw [fill](4,0) circle (2.5pt);
    \draw [fill](5,0) circle (2.5pt);
    \draw [fill](6,0) circle (2.5pt);
    \draw [fill](7,0) circle (2.5pt);
    \draw [line width = 2] (0,0) to [out = 270, in = 270] (3,0);
    \draw [line width = 2] (3,0) to [out = 90, in = 90] (4,0);
    \draw [line width = 2] (4,0) to [out = 270, in = 270] (5,0);
    \draw [line width = 2] (5,0) to [out = 90, in = 90] (2,0);
    \draw [line width = 2] (2,0) to [out = 270, in = 270] (1,0);
    \draw [line width = 2] (1,0) to [out = 90, in = 90] (6,0);
    \draw [line width = 2] (6,0) to [out = 270, in = 270] (7,0);
    \draw [line width = 2] (7,0) to [out =90, in = 90] (0,0);
    \node[above left] at (0,0){$0$};
    \node[above left] at (1,0){$1$};
    \node[above left] at (2,0){$2$};
    \node[above left] at (3,0){$3$};
    \node[above right] at (4,0){$4$};
    \node[above right] at (5,0){$5$};
    \node[above right] at (6,0){$6$};
    \node[above right] at (7,0){$0$};
  \end{tikzpicture}
  \caption{An Example of a closed meander.}\label{meandr1}
\end{figure}
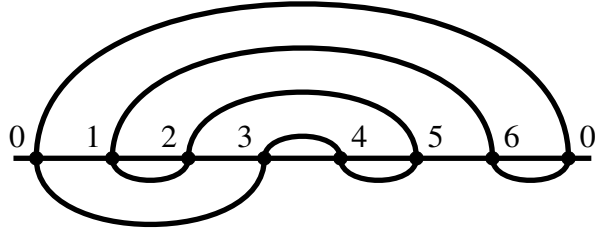

\begin{definition}\label{meandrmatrix}
  A matrix $\mathbf{M} = (\mathbf{m}_{ij}) \in \mathsf{Mat}_{n+1}(\mathbb{Z}_2)$ is called meander matrix if the following conditions hold:
  \begin{itemize}
    \item[(1)] its main diagonal contains only $0$,
    \item[(2)] $\mathbf{M}$ is symmetric,
    \item[(3)] it has at leas one string $\mathbf{m}_i = (\mathbf{m}_{i1}, \ldots, \mathbf{m}_{i, n+1})$ with $\mathbf{m}_{ik} = 1$ for every $1 \le k\le n+1$, $k \ne i$,
    \item[(4)] $\mathbf{M}$ satisfies the conditions of Theorem \ref{Gauss->matrix},
    \item[(5)] if $\mathbf{m}_{ij}=1$ and $\mathbf{m}_{jk} = 1$ then $\mathbf{m}_{ik}=1$ for all $1 \le i,j,k \le n+1$,
    \item[(6)] if $\mathbf{m}_{ik}=1$, then $\mathbf{m}_{ij}=1$ or $\mathbf{m}_{jk}=1$ for every $j$ with $1 \le i<j<k \le n+1$.
  \end{itemize}
  \end{definition}

\begin{figure}
  \begin{tikzpicture}[line width = 1.5]
   \draw[line width =3] (0,0) circle (2);
   \draw (0:2) -- (180:2);
   \fill (0:2) circle (3pt) node[right] {$0$};
   \fill (160:2) circle (3pt) node[above] {$1$};
   \fill (130:2) circle (3pt) node[above] {$2$};
   \fill (100:2) circle (3pt) node[above] {$3$};
   \fill (80:2) circle (3pt) node[above] {$4$};
   \fill (60:2) circle (3pt) node[above] {$5$};
   \fill (30:2) circle (3pt) node[above] {$6$};
   \fill (180:2) circle (3pt) node[left] {$0$};
   \fill (200:2) circle (3pt) node[left] {$6$};
   \fill (220:2) circle (3pt) node[below] {$1$};
   \fill (260:2) circle (3pt) node[below] {$2$};
   \fill (280:2) circle (3pt) node[below] {$5$};
   \fill (310:2) circle (3pt) node[below] {$4$};
   \fill (340:2) circle (3pt) node[below] {$3$};
   \draw (160:2) -- (220:2);
   \draw (130:2) -- (260:2);
   \draw (100:2) -- (340:2);
   \draw (80:2) -- (310:2);
   \draw (60:2) -- (280:2);
   \draw (160:2) -- (220:2);
   \draw (30:2) -- (200:2);
  \end{tikzpicture}
  \caption{A Gauss diagram of the meander, which is shown in {\sc Figure .\ref{meandr1}}, is shown}\label{GDm1}
\end{figure}
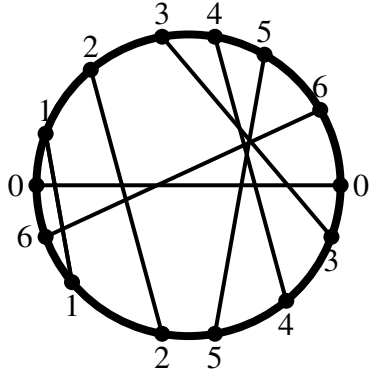

\begin{proposition}\label{matrix=meandr}
  Every meander matrix determines a unique closed meander (up to the equivalence).
\end{proposition}
\begin{proof}
  Let $\mathbf{M}$ be a meander matrix. Take $n+1$ lines on the plane $\mathbb{R}^2$ and mark them by numbers $1,2,\ldots,n+1$. Conditions (1) and (2) imply that the matrix $\mathbf{M}$ can be considered as a matrix of intersection of the lines: $\mathbf{m}_{ij} = 1$ if the lines with numbers $i$, $j$ intersect and $\mathbf{m}_{ij} = 0$ in otherwise. Further, from condition (3) it follows that there is a line is marked by, say $r$, such that it crossed by all other lines. Next, by Lemmas \ref{criteria}, \ref{9.1.11}, and by conditions (5), (6), all lines with exception of the line is marked by $r$ can be considered as a non-repeating braid. This gives rise a Gauss diagram (see {\sc Figure} \ref{GDm1}. Condition (4) implies that this Gauss diagram is releasable by a plane curve, say $\mathscr{C}$. Finally, the chord, which is marked by $r$, correspondences to a loop $\mathscr{C}(r)$ of $\mathscr{C}$. It is obviously that this loop contains all crossing of the curve $\mathscr{C}$. We then get the meander $\mathscr{C}$ up to the equivalence, as claimed.
\end{proof}

\section{Construction of Meanders}

We start with the following example.

\begin{example}
  Let us consider the meander is shown in {\sc Figure} \ref{meandr1}, its Gauss diagram is shown in {\sc Figure} \ref{GDm1}. The correspondence meander matrix has the following form

  \[
 \bordermatrix{
  &   0           & 1 & 2 & 3 & 4 & 5 & 6 \cr
  0 &  \mathbf{0} & 1 & 1 & 1 & 1 & 1 & 1 \cr
  1 &  1          &\mathbf{0} & 0& 0 & 0 & 0 & 1 \cr
  2 & 1 & 0 & \mathbf{0} & 0 & 0 & 0 & 1 \cr
  3 & 1 & 0 & 0 & \mathbf{0} & 1 & 1 & 1 \cr
  4 & 1 & 0 & 0 & 1 & \mathbf{0} &  1 & 1 \cr
  5 & 1 &0 &0 & 1 & 1 & \mathbf{0} & 1 \cr
  6 & 1 &1 &1 & 1 & 1 & 1 & \mathbf{0} \cr
  }
\]

This matrix is constructed as follows. Take a table with seven rows and seven columns. Number its columns from left to right and rows from top to bottom by numbers $0,1,\ldots, 6$. Fill the main diagonal by zeros.

\begin{center}
\begin{tabular}[t]{|c|c|c|c|c|c|c|c|}
 \hline
 & 0 & 1 &2 &3 &4 &5 & 6\\
 \hline
 0 & $\mathbf{0}$ & 1 & 1 & 1 &1 &1 &1 \\
 \hline
 1 & 1 & $\mathbf{0}$ & &&&& \\
 \hline
 2 & 1 && $\mathbf{0}$ &&&& \\
 \hline
 3 & 1&&& $\mathbf{0}$ &&& \\
 \hline
 4 & 1&&&& $\mathbf{0}$ && \\
 \hline
 5 & 1&&&&& $\mathbf{0}$ & \\
 \hline
 6 & 1&&&&&& $\mathbf{0}$  \\
 \hline
\end{tabular}
\end{center}

Since we start from $0$ and go then to $3$, then the chord $3$ intersects chords $0,3,5,6$, and hence we get
\begin{center}
\begin{tabular}[t]{|c|c|c|c|c|c|c|c|}
 \hline
 & 0 & 1 &2 &3 &4 &5 & 6\\
 \hline
 0 & $\mathbf{0}$ & 1 & 1 & 1 &1 &1 &1 \\
 \hline
 1 & 1 & $\mathbf{0}$ & & 0 & & & \\
 \hline
 2 & 1 && $\mathbf{0}$ & 0&&& \\
 \hline
 3 & 1 & 0 & 0 & $\mathbf{0}$ & 1 & 1& 1 \\
 \hline
 4 & 1&&&1& $\mathbf{0}$ && \\
 \hline
 5 & 1&&&1&& $\mathbf{0}$ & \\
 \hline
 6 & 1&&&1&&& $\mathbf{0}$  \\
 \hline
\end{tabular}
\end{center}

Next, we go to $4$ and then the chord $4$ crosses chords $0,3,5,6$ and we get
\begin{center}
\begin{tabular}[t]{|c|c|c|c|c|c|c|c|}
 \hline
 & 0 & 1 &2 &3 &4 &5 & 6\\
 \hline
 0 & $\mathbf{0}$ & 1 & 1 & 1 &1 &1 &1 \\
 \hline
 1 & 1 & $\mathbf{0}$ & & 0 & 0& & \\
 \hline
 2 & 1 && $\mathbf{0}$ & 0&0&& \\
 \hline
 3 & 1 & 0 & 0 & $\mathbf{0}$ & 1 & 1& 1 \\
 \hline
 4 & 1& 0&0 &1& $\mathbf{0}$ &1&1 \\
 \hline
 5 & 1&&&1&1& $\mathbf{0}$ & \\
 \hline
 6 & 1&&&1&1&& $\mathbf{0}$  \\
 \hline
\end{tabular}
\end{center}
and {\it etc.} Note that we changed the parity of numbers of strings.
\end{example}

\begin{lemma}\label{lemmaaboutfilling}
  To make by step-by-step a meander matrix, the following hold:
  \begin{itemize}
    \item[(1)] every string is filled as follows: every its empty cell is to the left of the main diagonal is filled by $0$ and every its empty cell is to the right of the main diagonal is filled by $1$,
    \item[(2)] to fill the next string its number has to have another parity than a number of the previous string.
  \end{itemize}
\end{lemma}

\begin{proof}
  Indeed, assume that we go from $i$ to $j$ in a meander. If they have the same parity then there are odd number of points between them {\sc Figure.\ref{oddbetween}}. This implies that the trajectory has to have self intersection points, {\it i.e.,} we get a no meander. Hence the numbers $i$, $j$ have to have different parity. This gives condition (2).

  \begin{figure}
  \begin{tikzpicture}
    \draw [line width = 2] (-0.3,0) to  (4.3,0);
    \draw [fill](0,0) circle (2.5pt);
    \draw [fill](1,0) circle (2.5pt);
    \draw [fill](2,0) circle (2.5pt);
    \draw [fill](3,0) circle (2.5pt);
    \draw [fill](4,0) circle (2.5pt);
    \draw [line width = 2] (0,0) to [out = 270, in = 270] (4,0);
    \draw [line width = 2] (0, 0.3) to (0,0);
    \draw [line width = 2] (4, 0.3) to (4,0);
    \node[below left] at (0,0){$i$};
    \node[below right] at (4,0){$j$};
  \end{tikzpicture}
  \caption{If $i,j$ have the same parity there are odd number of points between them then.}\label{oddbetween}
\end{figure}
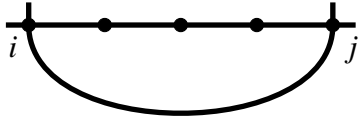

  Next, as we have already seen in the previous example, after choosing a string with number, say, $i$, the corresponding chord has to intersect chords with numbers $0,i+1,\ldots, n+1$, here $n$ is the number of all chords. This gives condition (1).
\end{proof}

\begin{definition}
  A string which is filled as in condition (1) of Lemma \ref{lemmaaboutfilling} is called $\Delta$-filled.
\end{definition}

\begin{lemma}\label{equations}
  Let in a step-by-step filling of a matrix a string with number $i$ is $\Delta$-filled. Then
  \[
   \bigl< \mathsf{i,j}  \bigr> \equiv 1 \bmod(2),
  \]
  for any $j>i$.
\end{lemma}
\begin{proof}
  Indeed, by Lemma \ref{lemmaaboutfilling}, $\mathbf{m}_{ij}=1$ for every $j>i$ in the matrix $\mathbf{M}$, and then by condition (3) of Theorem \ref{Gauss->matrix}, $\bigl<\mathsf{0,i} \bigr> + \bigl< \mathsf{0,j} \bigr> + \bigl<\mathsf{i,j} \bigr> \equiv 1 \bmod (2)$. We thus get
    \[
     \bigl<\mathsf{0,i} \bigr> + \bigl< \mathsf{0,j} \bigr> + \bigl<\mathsf{i,j} \bigr>= \bigl< \mathsf{i,j}\bigr>,
    \]
    because of $\bigl<\mathsf{0,i} \bigr> = \bigl<\mathsf{i,i} \bigr> - 1$, $1 \le i \le n$, as claimed.
\end{proof}

\begin{theorem}
  Let $\mathbf{M} \in \mathsf{Mat}_{n}(\mathbb{Z}_2)$ be a symmetric matrix with zero main diagonal. $\mathbf{M}$ is a meander matrix if and only if the following hold:
  \begin{itemize}
    \item[(1)] there is a string $\mathbf{m}_i$ such that $\mathbf{m}_{ij}=1$ if $j \ne i$ and $\mathbf{m}_{ii}=0$,
    \item[(2)] if $\mathbf{m}_{ij} = 1$, $\mathbf{m}_{jk} =1$, then $\mathbf{m}_{ik}=1$, for all $i,j,k \in \{1,2, \ldots, n\}$,
    \item[(3)] if $\mathbf{m}_{ik}=1$ then $\mathbf{m}_{ij}=1$ or $\mathbf{m}_{jk}=1$ for all $i \le j \le k$,
    \item[(4)] for every $1 \le i,j\le n$ we have
    \begin{eqnarray}
      \bigl< \mathbf{m}_i, \mathbf{m}_j \bigr> &\equiv&  0 \bmod(2), \qquad \mbox{if $i=j$ or $\mathbf{m}_{ij} = 0$},\label{<i,j>1}\\
      \bigl< \mathbf{m}_i, \mathbf{m}_j \bigr> &\equiv&  1 \bmod(2), \qquad \mbox{if $\mathbf{m}_{ij} = 1$}.\label{<i,j>2}
    \end{eqnarray}
      \end{itemize}
\end{theorem}
\begin{proof}
  By Definition \ref{meandrmatrix} and Proposition \ref{matrix=meandr}, it is sufficient to prove that condition (3) of Theorem \ref{Gauss->matrix} is equivalent to the last condition. It is clear that an equality $\bigl< \mathbf{m}_i, \mathbf{m}_i \bigr> \equiv  0 \bmod(2)$ implies that only an even number (including zero) of chords intersect chord correspondences to string $\mathbf{m}_i$. Next, an equality $\mathbf{m}_{ij}=0$ implies that chord $i$ do not intersect chord $j$, and an equality $\mathbf{m}_{ij}=0$ implies that only even number of chords intersect both of chords $i$, $j$. We thus get the first two conditions of Theorem \ref{Gauss->matrix}.

  Further, by Lemma \ref{equations}, from the third condition of Theorem \ref{Gauss->matrix} it follows (\ref{<i,j>1}). Assume now that (\ref{<i,j>1}) holds and consider three pairwise intersecting chords, say $i,j,k$. Then $\bigl< \mathsf{m}_i, \mathsf{m}_j \bigr> \equiv  1 \bmod(2)$, $\bigl< \mathbf{m}_i, \mathbf{m}_k \bigr> \equiv  1 \bmod(2)$, and $\bigl< \mathbf{m}_j, \mathbf{m}_k \bigr> \equiv  1 \bmod(2)$ and we complete the proof.
\end{proof}

From this Theorem it follows the following meanders construction algorithm.

\begin{itemize}
  \item[\sf{REQUIRE}] an even number $N$, $S_1 = \{1,3,\ldots, N-1\}$, $S_0 = \{2,4, \ldots,N\}$;
  \item[\sf{ENSURE}] $n_1,\ldots, n_N \in S_1 \cup S_0$;
  \item[{\bf{1}}.] take a $(N+1) \times (N+1)$ tableau with empty cells and fill the main diagonal by zeros;
  \item[{\bf{2}}.] number strings from left to right, and number columns from top to bottom by numbers $0,1, \ldots, N$;
  \item[{\bf{3}}.] $\Delta$-fill string with number $0$;
  \item[{\bf{4}}.] choose a string with an odd number $n\in S_1$ and $\Delta$-fill it and column with the same number;
  \item[{\bf{5}}.] $i=1$;
  \item[{\bf{6}}.] {\sf{PRINT}} $n$;
  \item[{\bf{6}}.] $S_i: = S_i \setminus \{n\}$;
  \item[{\bf{7}}.] {\sf{IF}} $S_i = \varnothing$ {\sf{THEN GOTO}} {\bf{13}} {\sf{ELSE GOTO}} {\bf{8}};
  \item[{\bf{8}}] $i:=i+1 \bmod (2)$;

  \item[{\bf{9}}.] choose a string (a column) with number $m \in S_i$;
  \item[{\bf{10}}.] {\sf{IF}} the string (the column) can be $\Delta$-filled {\sf{THEN}} {\sf{GOTO}} {\bf{11}} {\sf{ELSE}} choose another $m' \in S_i\setminus\{m\}$ {\sf{GOTO}} {\bf{10}};
  \item[{\bf{11}}.] using (\ref{<i,j>1}), (\ref{<i,j>2}) get a system of equations for empty cells;

  \item[{\bf{12}}.] {\sf{IF}} the system can by solved {\sf{THEN GOTO}} {\bf{6}} {\sf{ELSE}} choose a string (column) with another number $m' \in S_i\setminus \{m\}$  {\sf{GOTO}} {\bf{10}};
  \item[{\bf{13}}.] {\sf{END}}
  \end{itemize}

\begin{example}
  Let us construct a $9\times 9$ meander matrix.
  
  \begin{itemize}
    \item[(0)] Take $9\times 9$ tableau and partially fill it as follows
     \begin{center}
         \begin{tabular}[t]{|c|c|c|c|c|c|c|c|c|c|}
         \hline
         & 0 & 1 &2 &3 &4 &5 & 6 & 7 & 8 \\
         \hline
         0 & $\mathbf{0}$ & 1 & 1 & 1 &1 &1 &1 & 1 & 1  \\
         \hline
        1 &1  & $\mathbf{0}$ & &  & & & && \\
        \hline
        2 & 1 && $\mathbf{0}$ & &&& && \\
        \hline
        3 & 1 &  &  & $\mathbf{0}$ &  & & &  &  \\
        \hline
        4 & 1& & && $\mathbf{0}$ && & &   \\
        \hline
        5 & 1&&&&& $\mathbf{0}$ & && \\
        \hline
        6 & 1&&&&&& $\mathbf{0}$ &&  \\
        \hline
        7 & 1&&&& &&& $\mathbf{0}$ &  \\
        \hline
        8 & 1&&&&&&&& $\mathbf{0}$  \\
        \hline
        \end{tabular}
         \end{center}
        
     \item[(1)] Choose a string and column $5$ and $\Delta$-fill them
     
     \begin{center}
         \begin{tabular}[t]{|c|c|c|c|c|c|c|c|c|c|}
         \hline
         & 0 & 1 &2 &3 &4 &5 & 6 & 7 & 8 \\
         \hline
         0 & $\mathbf{0}$ & 1 & 1 & 1 &1 &1 &1 & 1 & 1  \\
         \hline
         1 &1  & $\mathbf{0}$ & & &  & 0& & & \\
         \hline
         2 & 1 && $\mathbf{0}$ &  &  &0 & && \\
         \hline
         3 & 1 &  &  & $\mathbf{0}$ &  &0 & &  &  \\
         \hline
         4 & 1& & && $\mathbf{0}$ & 0 & & &   \\
         \hline
         5 & 1&0&0&0& 0 & $\mathbf{0}$ &1 &1&1 \\
         \hline
         6 & 1&&&&&1& $\mathbf{0}$ &&  \\
         \hline
         7 & 1&&&&&1&& $\mathbf{0}$ &  \\
         \hline
         8 & 1&&&&&1&&& $\mathbf{0}$  \\
         \hline
        \end{tabular}
         \end{center}
    By(\ref{<i,j>2}),
    \[
    \begin{cases}
       \bigl< 5, 6 \bigr> = 1 + (6,7) + (6,8) \equiv 1 \bmod(2),\\
       \bigl< 5, 7 \bigr> = 1 + (6,7) + (7,8) \equiv 1 \bmod(2),\\
       \bigl< 5, 8 \bigr> = 1 + (6,8) + (7,8) \equiv 1 \bmod(2),
    \end{cases}
    \]
    it implies that
     \[
     (6,7) = (6,8) = (7,8) = {a} \in \{0,1\},
   \]
    and we thus get
    \begin{center}
         \begin{tabular}[t]{|c|c|c|c|c|c|c|c|c|c|}
         \hline
         & 0 & 1 &2 &3 &4 &5 & 6 & 7 & 8 \\
         \hline
         0 & $\mathbf{0}$ & 1 & 1 & 1 &1 &1 &1 & 1 & 1  \\
         \hline
         1 &1  & $\mathbf{0}$ & & &  & 0& & & \\
         \hline
         2 & 1 && $\mathbf{0}$ &  &  &0 & && \\
         \hline
         3 & 1 &  &  & $\mathbf{0}$ &  &0 & &  &  \\
         \hline
         4 & 1& & && $\mathbf{0}$ & 0 & & &   \\
         \hline
         5 & 1&0&0&0& 0 & $\mathbf{0}$ &1 &1&1 \\
         \hline
         6 & 1&&&&&1& $\mathbf{0}$ & ${a}$ & ${a}$  \\
         \hline
         7 & 1&&&&&1& ${a}$ & $\mathbf{0}$ & ${a}$ \\
         \hline
         8 & 1&&&&&1& ${a}$ & ${a}$ & $\mathbf{0}$  \\
         \hline
        \end{tabular}
         \end{center}
    
    \item[(2)] We have to choose a string (column) with an even number.
     \begin{itemize}
       \item[(i)] Take a string (column) $2$, we then get
                 \[
                   (1,2) = 0, (2,3) = (2,4) =0, (2,6) = (2,7) = (2,8)=1,
                  \]
                 it follows that $\bigl<\mathsf{1,2} \bigr> \equiv 0 \bmod (2)$. We have
                 \begin{eqnarray*}
                 \bigl<\mathsf{1,2} \bigr> &=& 1 + (1,3) + (1,4) \\
                 &&+ (1,6)+ (1,7) + (1,8) \equiv 0 \bmod (2),
                 \end{eqnarray*}
                 and we do not get any contradiction, hence this string may be chosen.
      \item[(ii)] Similarly, by the straightforward verification, one can easy verify that
                  strings $4$, $6$ can be chosen.
      \item[(iii)] Take string $8$. We get
                 \[
                 (1,8) = (2,8) = (3,8) = (4,8) = 0= {a} = 0,
                \]
                and hence $\bigl<\mathsf{1,8} \bigr> \equiv 0 \bmod(2)$, but the tableau  implies that $\bigl<\mathsf{1,8} \bigr> = 1$. Therefore this string cannot be chosen in this step.
      \end{itemize}
    
      \item[(3)] Chose string $2$. We then get
      \begin{center}
         \begin{tabular}[t]{|c|c|c|c|c|c|c|c|c|c|}
         \hline
         & 0 & 1 &2 &3 &4 &5 & 6 & 7 & 8 \\
         \hline
         0 & $\mathbf{0}$ & 1 & 1 & 1 &1 &1 &1 & 1 & 1  \\
         \hline
         1 &1  & $\mathbf{0}$ & 0 & &  & 0& & & \\
         \hline
         2 & 1 &0& $\mathbf{0}$ & 1 & 1 &0 &1 &1&1 \\
         \hline
         3 & 1 &  & 1 & $\mathbf{0}$ &  &0 & &  &  \\
         \hline
         4 & 1& &1 && $\mathbf{0}$ & 0 & & &   \\
         \hline
         5 & 1&0&0&0& 0 & $\mathbf{0}$ &1 &1&1 \\
         \hline
         6 & 1&&1&&&1& $\mathbf{0}$ & ${a}$ & ${a}$  \\
         \hline
         7 & 1&&1&&&1& ${a}$ & $\mathbf{0}$ & ${a}$ \\
         \hline
         8 & 1&&1&&&1& ${a}$ & ${a}$ & $\mathbf{0}$  \\
         \hline
        \end{tabular}
    \end{center}
    by (\ref{<i,j>2}),
    \[
     \begin{cases}
     \bigl< \mathsf{2,3} \bigr> = 1 + (3,4) + (3,6) + (3,7) + (3,8) \equiv 1 \bmod(2),\\
     \bigl< \mathsf{2,4} \bigr> = 1 + (3,4) + (4,6) + (4,7) + (4,8) \equiv 1 \bmod(2),\\
     \bigl< \mathsf{2,6} \bigr> = 1 + (3,6) + (4,6) + {a} + {a} \equiv 1 \bmod(2),\\
     \bigl< \mathsf{2,7} \bigr> = 1 + (3,7) + (4,7) + {a} + {a} \equiv 1 \bmod(2),\\
     \bigl< \mathsf{2,8} \bigr> = 1 + (3,8) + (4,8) + {a} + {a} \equiv 1 \bmod(2),
     \end{cases}
    \]
    and we then obtain
    \begin{align*}
      & (3,6)  = (4,6) = b \in \{0,1\} \\
      & (3,7)  = (4,7) = c \in \{0,1\} \\
      & (3,8)  = (4,8) = d \in \{0,1\} \\
      & (3,4)  = b+c+d \bmod(2).
    \end{align*}
    
    Next, using the condition of Theorem \ref{Gauss->matrix}, $\bigl< \mathsf{3,3} \bigr> \equiv 0 \bmod(2)$, we get $(1,3) = 0$. Similarly, one can get $(1,4) = 0$. Using
    \[
     \bigl< \mathsf{6,6} \bigr> \equiv \bigl< \mathsf{7,7} \bigr> \equiv \bigl< \mathsf{8,8} \bigr> \equiv 0 \bmod(2)
    \]
    we get $(1,6) = (1,7) = (1,8) = 1$.
    
    We thus have
    \begin{center}
         \begin{tabular}{|c| c | c |c|c|c|c|c|c|c|}
         \hline
         & $\quad 0 \quad $ & $\quad 1 \quad $ & $\quad 2 \quad $ & $\quad 3 \quad $ &  $\quad 4 \quad $ & $\quad 5 \quad $ & $\quad 6 \quad $ & $\quad 7 \quad $ & $\quad 8 \quad $ \\
         \hline
         0 & $\mathbf{0}$ & 1 & 1 & 1 &1 &1 &1 & 1 & 1  \\
         \hline
         1 &1  & $\mathbf{0}$ & 0 & 0 & 0 & 0&1 &1 &1 \\
         \hline
         2 & 1 &0& $\mathbf{0}$ & 1 & 1 &0 &1 &1&1 \\
         \hline
         3 & 1 & 0 & 1 & $\mathbf{0}$ &  &0 & $b$ & $c$  & $d$ \\
         \hline
         4 & 1& 0 & 1 & & $\mathbf{0}$ & 0 & $b$ & $c$ & $d$  \\
         \hline
         5 & 1&0&0&0& 0 & $\mathbf{0}$ &1 &1&1 \\
         \hline
         6 & 1 & 1 &1& $b$ & $b$ &1& $\mathbf{0}$ & ${a}$ & ${a}$  \\
         \hline
         7 & 1& 1 &1& $c$ & $c$ &1& ${a}$ & $\mathbf{0}$ & ${a}$ \\
         \hline
         8 & 1& 1 &1& $d$ & $d$ &1& ${a}$ & ${a}$ & $\mathbf{0}$  \\
         \hline
        \end{tabular}
    \end{center}
    
    \item[(4)] We have to chose a string with an odd number $1,3$ or $7$. Since string $4$ has not been chosen and $(1,4)=0$ then string $1$ cannot be $\Delta$-filled.
        
        Next, if we chose string $7$ we then get $(6,7) = 0$, $(7,8) = 1$, but $(6,7) = (7,8) = a$, we then get a contradiction. Further, one can easy verify that string $3$ can be chosen.
    
    \item[(5)] Take string $3$. It follows that $b=c=d=1$ and we then get
    \begin{center}
         \begin{tabular}{|c| c | c |c|c|c|c|c|c|c|}
         \hline
         & $\quad 0 \quad $ & $\quad 1 \quad $ & $\quad 2 \quad $ & $\quad 3 \quad $ &  $\quad 4 \quad $ & $\quad 5 \quad $ & $\quad 6 \quad $ & $\quad 7 \quad $ & $\quad 8 \quad $ \\
         \hline
         0 & $\mathbf{0}$ & 1 & 1 & 1 &1 &1 &1 & 1 & 1  \\
         \hline
         1 &1  & $\mathbf{0}$ & 0 & 0 & 0 & 0&1 &1 &1 \\
         \hline
         2 & 1 &0& $\mathbf{0}$ & 1 & 1 &0 &1 &1&1 \\
         \hline
         3 & 1 & 0 & 1 & $\mathbf{0}$ & 1 &0 & $1$ & $1$  & $1$ \\
         \hline
         4 & 1& 0 & 1 & 1 & $\mathbf{0}$ & 0 & $1$ & $1$ & $1$  \\
         \hline
         5 & 1&0&0&0& 0 & $\mathbf{0}$ &1 &1&1 \\
         \hline
         6 & 1 & 1 &1& $1$ & $1$ &1& $\mathbf{0}$ & ${a}$ & ${a}$  \\
         \hline
         7 & 1& 1 &1& $1$ & $1$ &1& ${a}$ & $\mathbf{0}$ & ${a}$ \\
         \hline
         8 & 1& 1 &1& $1$ & $1$ &1& ${a}$ & ${a}$ & $\mathbf{0}$  \\
         \hline
        \end{tabular}
    \end{center}
        
        We see that strings $1,4$ are automatically $\Delta$-filled.
        
        So, we have two possibilities: 1) $a=0$, and 2) $a = 1$. These cases correspondence to the following possibilities for our meander (see {\sc Figure} \ref{exmeandr}).
      \end{itemize}
  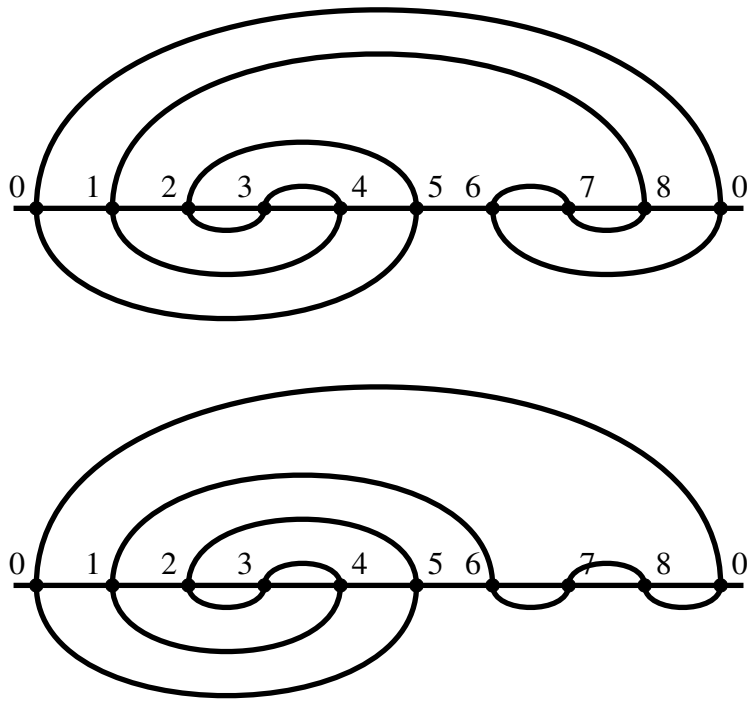
\begin{figure}[h!]
  \begin{tikzpicture}
    \draw [line width = 2] (-0.3,0) to  (9.3,0);
    \draw [fill](0,0) circle (2.5pt);
    \draw [fill](1,0) circle (2.5pt);
    \draw [fill](2,0) circle (2.5pt);
    \draw [fill](3,0) circle (2.5pt);
    \draw [fill](4,0) circle (2.5pt);
    \draw [fill](5,0) circle (2.5pt);
    \draw [fill](6,0) circle (2.5pt);
    \draw [fill](7,0) circle (2.5pt);
    \draw [fill](8,0) circle (2.5pt);
    \draw [fill](9,0) circle (2.5pt);
    \draw [line width = 2] (0,0) to [out = 270, in = 270] (5,0);
    \draw [line width = 2] (5,0) to [out = 90, in = 90] (2,0);
    \draw [line width = 2] (2,0) to [out = 270, in = 270] (3,0);
    \draw [line width = 2] (3,0) to [out = 90, in = 90] (4,0);
    \draw [line width = 2] (4,0) to [out = 270, in = 270] (1,0);
    \draw [line width = 2] (1,0) to [out = 90, in = 90] (8,0);
    \draw [line width = 2] (8,0) to [out = 270, in = 270] (7,0);
    \draw [line width = 2] (7,0) to [out = 90, in = 90] (6,0);
    \draw [line width = 2] (6,0) to [out = 270, in = 270] (9,0);
    \draw [line width = 2] (9,0) to [out = 90, in = 90] (0,0);
    \node[above left] at (0,0){$0$};
    \node[above left] at (1,0){$1$};
    \node[above left] at (2,0){$2$};
    \node[above left] at (3,0){$3$};
    \node[above right] at (4,0){$4$};
    \node[above right] at (5,0){$5$};
    \node[above left] at (6,0){$6$};
    \node[above right] at (7,0){$7$};
    \node[above right] at (8,0){$8$};
    \node[above right] at (9,0){$0$};
\begin{scope}[yshift = -5cm]
\draw [line width = 2] (-0.3,0) to  (9.3,0);
    \draw [fill](0,0) circle (2.5pt);
    \draw [fill](1,0) circle (2.5pt);
    \draw [fill](2,0) circle (2.5pt);
    \draw [fill](3,0) circle (2.5pt);
    \draw [fill](4,0) circle (2.5pt);
    \draw [fill](5,0) circle (2.5pt);
    \draw [fill](6,0) circle (2.5pt);
    \draw [fill](7,0) circle (2.5pt);
    \draw [fill](8,0) circle (2.5pt);
    \draw [fill](9,0) circle (2.5pt);
    \draw [line width = 2] (0,0) to [out = 270, in = 270] (5,0);
    \draw [line width = 2] (5,0) to [out = 90, in = 90] (2,0);
    \draw [line width = 2] (2,0) to [out = 270, in = 270] (3,0);
    \draw [line width = 2] (3,0) to [out = 90, in = 90] (4,0);
    \draw [line width = 2] (4,0) to [out = 270, in = 270] (1,0);
    \draw [line width = 2] (1,0) to [out = 90, in = 90] (6,0);
    \draw [line width = 2] (6,0) to [out = 270, in = 270] (7,0);
    \draw [line width = 2] (7,0) to [out = 90, in = 90] (8,0);
    \draw [line width = 2] (8,0) to [out = 270, in = 270] (9,0);
    \draw [line width = 2] (9,0) to [out = 90, in = 90] (0,0);
    \node[above left] at (0,0){$0$};
    \node[above left] at (1,0){$1$};
    \node[above left] at (2,0){$2$};
    \node[above left] at (3,0){$3$};
    \node[above right] at (4,0){$4$};
    \node[above right] at (5,0){$5$};
    \node[above left] at (6,0){$6$};
    \node[above right] at (7,0){$7$};
    \node[above right] at (8,0){$8$};
    \node[above right] at (9,0){$0$};
   \end{scope}
  \end{tikzpicture}
  \caption{The meanders which are correspondence to case $a=0$ (in the top) and $a=1$ (in the bottom).}\label{exmeandr}
\end{figure}

\end{example}

\end{document}